\documentclass[final,leqno]{siamltex}
\usepackage{mathbbol}
\usepackage{mathrsfs}
\usepackage{amsfonts}
\usepackage{stmaryrd}
\usepackage{amsmath}
\usepackage{amssymb}
\usepackage{multirow}
\usepackage{caption}
\usepackage{tabularx}
\usepackage{stmaryrd}
\usepackage{booktabs}
\usepackage{cite}
\usepackage{amssymb}
\usepackage{verbatim}
\usepackage{color}
\usepackage{graphicx}
\usepackage{subfigure}
\usepackage{hyperref}
\usepackage[all,cmtip]{xy}
\usepackage{epsfig,epstopdf,color,bm}
\usepackage{tabularx}
\allowdisplaybreaks[4]

\input psfig.sty

\newcommand{\p}{\partial}

\renewcommand{\theequation}{\arabic{section}.\arabic{equation}}
\newcommand{\be}{\begin{equation}}
\newcommand{\ee}{\end{equation}}
\newcommand{\ba}{\begin{array}}
\newcommand{\ea}{\end{array}}
\newcommand{\bea}{\begin{eqnarray}}
\newcommand{\eea}{\end{eqnarray}}
\newcommand{\beas}{\begin{eqnarray*}}
\newcommand{\eeas}{\end{eqnarray*}}
\newtheorem{remark}{Remark}[section]

\def\R{{\mathbb R}}% real numbers
\def\S{{\mathbb S}}% Sphere

\newcommand{\cT}{\mathcal T}
\newcommand{\cP}{\mathcal P}
\newcommand{\cE}{\mathcal E}
\newcommand\cN{{\mathcal N}}
\newcommand{\eps}{\varepsilon}

\renewcommand{\l}{\left}
\renewcommand{\r}{\right}

\renewcommand{\d}{\mathrm{d}}

\title{A convexity-preserving and perimeter-decreasing parametric finite element method for the area-preserving curve shortening flow}
\author{Wei Jiang\thanks{School of Mathematics and Statistics,
Wuhan University, Wuhan, 430072, China ({\tt jiangwei1007@whu.edu.cn}). This author's research was supported by the National Natural Science Foundation of China Nos. 12271414 and 11871384.}
    \and
Chunmei Su\thanks{Yau Mathematical Sciences Center, Tsinghua University, Beijing, 100084, China ({\tt sucm@tsinghua.edu.cn}).}
\and
Ganghui Zhang\thanks{Yau Mathematical Sciences Center, Tsinghua University, Beijing, 100084, China({\tt gh-zhang19@mails.tsinghua.edu.cn}).}}

\date{}
\begin{document}

\maketitle

\begin{abstract}
We propose and analyze a semi-discrete parametric finite element scheme for solving the area-preserving curve shortening flow. The scheme is based on Dziuk's approach (SIAM J. Numer. Anal. 36(6): 1808-1830, 1999) for the anisotropic curve shortening flow. We prove that the scheme preserves two fundamental geometric structures of the flow with an initially convex curve: (i) the convexity-preserving property, and (ii) the perimeter-decreasing property. To the best of our knowledge, the convexity-preserving property of numerical schemes which approximate the flow is rigorously proved for the first time. Furthermore, the error estimate of the semi-discrete scheme is established, and numerical results are provided to demonstrate the structure-preserving properties as well as the accuracy of the scheme.
\end{abstract}


\pagestyle{myheadings}\thispagestyle{plain}
\section{Introduction}
Consider the volume-preserving mean curvature flow driven by the motion law
\be\label{vp}
v=(H-\l<H\r>)\cN, \quad \mathrm{on}\quad \Gamma_t,\ee
where $\Gamma_t$ is a family of smooth hypersurfaces in $\mathbb{R}^n$, $v$ denotes the velocity, $\cN$ is the inner normal vector, $H$ represents the scalar mean curvature (with the sign convention that $H$ is positive for balls), and $\l<H\r>:={\int_{\Gamma_t} H ds^{n-1}}/{\int_{\Gamma_t}\,ds^{n-1}}$ is the average mean curvature along $\Gamma_t$. It is well-known that the volume-preserving mean-curvature flow can be interpreted as the $L^2$-gradient flow of the area functional under configurations with a fixed volume \cite{Mugnai}. The volume-preserving mean curvature flow has the following fundamental geometric properties, i.e.,
\begin{itemize}
\item [(i)] \emph{Volume-preserving} blue\cite[Lemma 5.25]{Andrews-Chow-Guenther-Langford}. It can be immediately verified that the volume enclosed by $\Gamma_t$ is indeed preserved by noticing
\[\frac{\d}{\d t}|\Omega_t|=-\int_{\Gamma_t} v\cdot \cN ds^{n-1}=-\int_{\Gamma_t} (H-\l<H\r>)ds^{n-1} =0,\]
where $\Omega_t$ is the region enclosed by $\Gamma_t$. In dimension two (i.e., $n=2$), it becomes the {\textbf{area-preserving}} property for a planar curve.

\item [(ii)] \emph{Area-shrinking} blue \cite[Lemma 5.25]{Andrews-Chow-Guenther-Langford}. Actually, one can easily check that
\[\frac{\d}{\d t}|\Gamma_t|=-\int_{\Gamma_t} H v\cdot \cN ds^{n-1}=-\int_{\Gamma_t}(H-\l<H\r>)^2ds^{n-1}\le 0.\]
When $n=2$, it becomes {\textbf{perimeter-decreasing}} for a planar curve.

\item [(iii)] \emph{Convexity-preserving}. When $n=2$, it was shown by Gage that starting from an initially, smooth and convex closed curve, this flow \eqref{vp} {\textbf{preserves the convexity}} and evolves the curve into a circle \cite{Gage85}. Furthermore, Huisken extended the result to higher dimensional cases \cite{Huisken87}. For more general initial data, interested readers may refer to \cite{Escher-Simonett,ATW}.
\end{itemize}

In this paper, we focus on the planar curve ($n=2$). In this case, the volume-preserving mean curvature flow is also known as
the area-preserving curve shortening flow (AP-CSF), and it can be parametrized by the following equation \cite{Gage85}
\begin{equation}\label{AP-CSF}
	\begin{cases}
		\p_tX=\l( H-\frac{2\pi}{L} \r)\cN,\quad &\xi\in \S^1,
\quad t\in (0,T],  \\
X(\xi,0)=X^0(\xi),\quad &\xi\in \S^1,
	\end{cases}
\end{equation}
where $X(\xi,t):\S^1\times [0,T]\rightarrow \Gamma_t\subseteq\R^2$, $L:=L(t)$ is the length of $\Gamma_t$, by recalling the theorem of turning tangents \cite{Carmo}, i.e., $\int_{\Gamma_t} Hds^1=2\pi$, for a simple closed curve $\Gamma_t$.

Nowadays, the AP-CSF has found important applications in many research areas, such as material science and image processing \cite{Italo-Stefano-Riccardo}, and it can be viewed as an area-preserving variant of the CSF \cite{Mayer,Taylor-Cahn} or  a limit flow of nonlocal
Ginzburg-Landau equation \cite{Bronsard-Stoth}. There have been extensive numerical investigations concerning with the CSF or AP-CSF in the last decades. Among them, parametric finite element methods (PFEMs) have been widely proposed for simulating the CSF and some other related geometric flows \cite{Bao-Garcke-Nurnberg-Zhao}, e.g., the surface diffusion flow \cite{Bao-Zhao}, and anisotropic geometric flows \cite{BGN2007_1,BGN2019,BGN2008_1,BGN2008_2}. Numerical approximations to the CSF by using PFEMs could date back to the pioneering work of Dziuk \cite{Dziuk1991} in 1991. Since then, various techniques have been introduced to make the designed PFEMs more
accurate and efficient in practical simulations, including the method of Barrett, Garcke and N\"urnberg (the BGN scheme) \cite{BGN2007_1, BGN2007_2} based on a novel variational formulation, the method
of Deckelnick and Dziuk by introducing an artificial tangential velocity
\cite{Deckelnick-Dziuk}, and the method proposed by Elliott and Fritz based on special
reparametrizations \cite{Elliott-Fritz}. These methods induce appropriate tangential motions that lead
to good mesh distribution property, which play a vital role in numerical simulations.
Recently, more and more attention has been paid to designing ``structure-preserving'' (e.g., area-preserving or perimeter-decreasing)
PFEMs for solving geometric evolution flows~\cite{Jiang21,Bao-Zhao,Bao-Garcke-Nurnberg-Zhao}.

However, error estimates for these schemes seem difficult and quite challenging. For example, Dziuk first studied the convergence of a semi-discrete linearly implicit PFEM for the CSF \cite{Dziuk1994} and anisotropic CSF \cite{Dziuk1999}, respectively, based on a finite difference structure; Li developed a new technique to analyze the convergence of semi-discrete high-order PFEMs for the CSF \cite{Li1} and mean curvature flow of closed surfaces \cite{Li2}, respectively. For Dziuk's fully discrete linearly implicit scheme~\cite{Dziuk1994}, until very recently, an optimal error estimate in $H^1$ has been established by Ye and Cui~\cite{Ye-Cui}. As for the error analysis about other numerical methods of the CSF or other related geometric flows, we refer to \cite{Deckelnick-Dziuk,Hu-Li,Kovacs-Li-Lubich2019,Kovacs-Li-Lubich2021, Pozzi-Stinner,Elliott-Fritz,Barrett-Deckelnick-Styles}.

Back to the AP-CSF, there exist various numerical methods in the literature, e.g., the finite difference method \cite{Mayer}, the MBO method \cite{Ruuth-Wetton,Catherine-Selim-Jeffrey}, the crystalline algorithm \cite{Ushijima-Yazaki} and PFEMs \cite{BGN2020, Pei-Li}. Particularly, structure-preserving properties were investigated in \cite{Ushijima-Yazaki, BGN2020, Pei-Li, Sakakibara-Miyatake}. For example, the semi-discrete PFEM in \cite{BGN2020} based on an elegant variational formulation was shown to preserve the length shortening property, and the fully discrete crystalline algorithm in \cite{Ushijima-Yazaki}, the semi-discrete polygonal evolution law in \cite{Sakakibara-Miyatake} based on the definition of tangent and normal vectors/velocities at each vertex, the fully discrete PFEM in \cite{Pei-Li} was shown to be area-preserving and perimeter-decreasing. However, error estimates have been barely studied for the above mentioned methods, except for the crystalline algorithm in \cite{Ushijima-Yazaki} where the error estimate was only established for the curvature since the numerical scheme was designed based on the crystalline approximation which merely involves the curvature. To the best of our knowledge, there exists few numerical analysis about numerical methods for solving the AP-CSF, and the convexity-preserving property has never been investigated in the literature. The reason lies in that the nonlocal term in the AP-CSF has brought troubles and considerable challenges in numerical analysis.

In this paper, we propose a semi-discrete PFEM for the AP-CSF based on Dziuk's approach for anisotropic CSF \cite{Dziuk1999}, investigate its structure-preserving properties and present its error analysis. Specifically, we prove that our scheme preserves two important geometric structures of the AP-CSF, i.e., convexity-preserving and perimeter-decreasing properties. As far as we know, this is the first job to rigorously prove the convexity-preserving property and to give the error estimate of numerical methods for solving the AP-CSF.

To start, \eqref{AP-CSF} can be written more explicitly as
\begin{equation}\label{AP-CSF, parametrization}
		 			\p_t X=\frac{1}{|\p_\xi X|}\p_\xi\Big(\frac{\p_\xi X}{|\p_\xi X|} \Big)-\frac{2\pi}{L}\Big(\frac{\p_\xi X}{|\p_\xi X|}\Big)^{\perp},
		 \end{equation}
		where $(a,b)^\perp:=(-b,a)$. This naturally yields a weak formulation: for any $v\in (H^1(\S^1))^2$, it holds
		\begin{equation}\label{AP-CSF, weak}
	\begin{split}
			&\int_{\S^1}|\p_\xi X|\p_t X\cdot v\ \d \xi+ \int_{\S^1} \frac{\p_\xi X}{|\p_\xi X|} \cdot \p_\xi v\ \d \xi +\int_{\S^1}\frac{2\pi }{L} (\p_\xi X)^\perp\cdot v \  \d \xi=0.
		\end{split}
\end{equation}
As mentioned in \cite{Dziuk1994}, the derived linearly implicit PFEM from the above formulation for the CSF (with the last term missing) may fail to preserve the length shortening property of the CSF. To overcome this, Dziuk proposed another scheme based on the lumping of masses in \cite{Dziuk1994} for the CSF. Here we utilize the similar approach: find a solution $X_h(\xi,t)\in V_h\times[0,T]$ satisfying the weak form \eqref{Semidiscrete, weak}
%\begin{equation}\label{fapp}
%		\begin{split}
%			&\int_{\S^1}|\p_\xi X_h|\p_t X_h\cdot v_h\ \d \xi+ \int_{\S^1} \frac{\p_\xi X_h}{|\p_\xi X_h|} \cdot \p_\xi v_h\ \d \xi\\
%		&\ +\int_{\S^1}\frac{\mathbf{h^2}|\p_\xi X_h|}{6}\p_\xi\p_t X_h\cdot \p_\xi v_h\mathrm{~d} \xi +\int_{\S^1}\frac{2\pi }{L_h} (\p_\xi X_h)^\perp\cdot v_h \  \d \xi=0,\quad \forall\ v_h\in V_h,
%		\end{split}
%		\end{equation}
with initial condition $X_h(\xi,0)=I_hX^0$, where $V_h$ is a vector valued Lagrange finite element space consisting of piecewise linear polynomial and $I_h$ is the standard Lagrange interpolation. Similar as in \cite{Chow-Glickenstein}, the semi-discrete scheme focuses on the motion of the initial polygon, which is determined by the evolution of the vertices. We show that if the initial curve is convex, then the evolved polygon keeps convex all the time.  Moreover, the perimeter of the polygon  is decreasing. To show the convexity-preserving property, we characterize the convexity of a polygon  by the positivity of the oriented area of all adjacent triangles, which will be shown by a contradiction argument. Surprisingly, the perimeter-decreasing property can be reduced to a pure trigonometric inequality when the polygon keeps convex. We note that nondegeneration of vertices is necessary to ensure the evolved polygons are well-behaved. This will be guaranteed by the error estimate of the scheme, which shows that the semi-discrete scheme \eqref{Semidiscrete, weak} converges in $H^1$ at the first order, and the lower bound of the edge lengths of the polygon could keep positive all the time.
%The key points of the error analysis lies in the following basic properties:
%\begin{align*}
%&\p_t|\p_\xi X | = -|\p_\xi X||\p_t X |^2-\frac{2\pi}{L}|\p_\xi X|\mathcal{N}\cdot \p_t X,\\
%&\left|A-B\right|^{2}=|A||B| |A/|A|-B/|B||^2+(|A|-|B|)^2,\quad A, B\in \mathbb{R}^2,\quad A\neq 0, \quad B\neq 0,
%\end{align*}
%and the difference of the nonlocal terms--perimeters can be controlled by length element difference
%\begin{equation}\label{Estimate of perimeter}
% \l|L-L_h\r|=\l|\int_{\S^1}|\p_\xi X|-|\p_\xi X_h|\ \d \xi\r| \le C\l(\int_{\S^1}\l(|\p_\xi X|-|\p_\xi X_h|\r)^2\ \d \xi\r)^{1/2}.
%\end{equation}

The rest of the paper is organized as follows. In Section 2, we start with the spatial discretization which approximates the AP-CSF and summarize our main results. In Section 3, we prove that the numerical scheme rigorously preserves two important geometric structures of the flow, i.e., the convexity-preserving and perimeter-decreasing properties. Then, we present the proof of the error estimate of the scheme in Section 4. Finally, some numerical results produced by the scheme are provided in Section 5 to validate our theoretical results.

\smallbreak

\section{Spatial discretization and main results}
Let $0=\xi_0<\xi_1<\ldots<\xi_N=2\pi$ be a partition of $\S^1=[0,2\pi]$. We denote $h_j=\xi_j-\xi_{j-1}$ by the length of the interval $I_j:=[\xi_{j-1}, \xi_j]$ and $h=\max\limits_{j} h_j$. Throughout the paper, we use a periodic index, i.e., $f_{j}=f_{j\pm N}$ when involved. We assume that the partition and the exact solution are regular in the following senses, respectively:
	  \smallskip

\textbf{(Assumption 2.1)} There exist constants $C_p$ and $C_P$ such that
	\[
	\min_{j} h_j\ge C_p h,\quad | h_{j+1}-h_j|\le C_P h^2,\quad 1\le j\le N.
	\]

	\textbf{(Assumption 2.2)} Suppose that the unique solution of \eqref{AP-CSF} with an initial value $X^0\in H^{2}(\S^1)$ satisfies $X\in W^{1,\infty}\l([0,T],H^2(\S^1)\r)$, i.e., \[K(X):=\|X\|_{W^{1,\infty}\l([0,T],H^2(\S^1)\r)}<\infty.\]
	 We further assume that there exist constants  $0<\kappa_1<\kappa_2$  such that
	  \[\kappa_1\le \l|\p_\xi X(\xi,t)\r|\le \kappa_2,\quad \forall\ (\xi, t)\in \S^1\times [0, T].
	  \]
	
	  \smallskip

	We define the following finite element space consisting of piecewise linear functions satisfying periodic boundary conditions:
	\[
   V_h=\l\{v\in C^0(\S^1,\R^2): v|_{I_j}\in P_1(I_j),\quad 1\le j\le N,\quad v(\xi_0)=v(\xi_N)\r\},
	\]
where $P_1$ denotes all polynomials with degrees at most $1$. For any continuous function $v\in C^0(\S^1,\R^2)$, the linear interpolation  $I_hv\in V_h$ is uniquely determined through $I_hv(\xi_j)=v(\xi_j)$ for all $1\le j\le N$ and can be explicitly written as $I_hv(\xi)=\sum\limits_{j=1}^N v(\xi_j)\varphi_j(\xi)$, where $\varphi_j$ represents the standard Lagrange basis function satisfying  $\varphi_j(\xi_i)=\delta_{ij}$. We have  the following basic estimates from finite element theory.

\smallskip
\begin{lemma}[\cite{Brenner-Scott}] Under Assumption 2.1, there exists a constant $C$ depending on $C_p, C_P$ such that the following estimates hold:\\
\noindent (i) (Interpolation estimate). For any $Y\in H^2(\S^1)$, we have
		\begin{equation}\label{Interpolation estimate}
			\begin{split}
				&\|Y-I_h Y\|_{L^2}
				\le Ch^k\|Y\|_{H^k},\quad k=1,2;\quad  \|Y-I_h Y\|_{L^\infty}
				\le Ch^{1/2}\|Y\|_{H^1},\\
 &				\|\p_\xi\l(Y-I_hY\r)\|_{L^2}
				\le Ch\|Y\|_{H^2}, \quad
				\|\p_\xi I_hY\|_{L^2}
			   \le C\|Y\|_{H^1}.
			\end{split}
		\end{equation}
\noindent (ii) (Inverse estimate). For $v_h\in V_h$, we have
		\begin{equation}\label{Inverse estimate}
			\|v_h\|_{L^\infty}\le Ch^{-1/2}\|v_h\|_{L^2},\quad \|v_h\|_{H^1}\le Ch^{-1}\|v_h\|_{L^2}.
		\end{equation}
\end{lemma}

\begin{definition}
	We call a function
\be\label{Xh}
X_h(\xi,t)=\sum_{j=1}^NX_j(t)\varphi_j(\xi):\S^1\times [0,T]\rightarrow \R^2
 \ee
 is a semidiscrete solution of (\ref{AP-CSF, parametrization}) if it satisfies the following weak formulation
			\begin{equation}\label{Semidiscrete, weak}
		\begin{split}
			&\int_{\S^1}|\p_\xi X_h|\p_t X_h\cdot v_h\ \d \xi+ \int_{\S^1} \frac{\p_\xi X_h}{|\p_\xi X_h|} \cdot \p_\xi v_h\ \d \xi\\
		&\ +\int_{\S^1}\frac{\mathbf{h}^2|\p_\xi X_h|}{6}\p_\xi\p_t X_h\cdot \p_\xi v_h\mathrm{~d} \xi +\int_{\S^1}\frac{2\pi }{L_h} (\p_\xi X_h)^\perp\cdot v_h \  \d \xi=0,\quad \forall\ v_h\in V_h,
		\end{split}
		\end{equation}
		with initial condition $X_h(\xi,0)=I_hX^0$, where $L_h$ represents the perimeter of the evolved curve (image of $X_h$), i.e.,
		\[
		L_h:=\sum_{j=1}^N h_j\l.|\p_\xi X_h| \r|_{I_j}=\sum_{j=1}^N|X_j-X_{j-1}| =:\sum_{j=1}^N q_j,
		\]
		and  $\mathbf{h}$ is a piecewise constant $\mathbf{h}=h_j$ on $I_j$.
\end{definition}
\begin{remark}
 A similar version of \eqref{Semidiscrete, weak} was proposed and analyzed in \cite[(9) and (15)]{Dziuk1994} and \cite[Definition 4.1]{Dziuk1999} for CSF and anisotropic CSF, respectively. The introduction of the third term $\int_{\mathbb{S}^1}\frac{\mathbf{h}^2|\partial_\xi X_h|}{6}\partial_\xi\partial_t X_h\cdot \partial_\xi v_h \mathrm{d}\xi$ in \eqref{Semidiscrete, weak} gives rise to the so-called mass-lumped scheme (similar to \eqref{Semidiscrete, lumped mass}) that preserves the length shortening property for the CSF, which was missing for the original formulation (e.g., \eqref{AP-CSF, weak}). On the other hand, a more natural explanation was given in \cite[(1.6) and (3.12)]{Pozzi-Stinner}, where it was shown that \eqref{Semidiscrete, weak} is equivalent to the following scheme
			\begin{equation*}
				\int_{\S^1}|\p_\xi X_h|I_h(\p_t X_h\cdot v_h)\ \d \xi+ \int_{\S^1} \frac{\p_\xi X_h}{|\p_\xi X_h|} \cdot \p_\xi v_h\ \d \xi+\int_{\S^1}\frac{2\pi }{L_h} (\p_\xi X_h)^\perp\cdot v_h \  \d \xi=0,\ \forall\ v_h\in V_h,
			\end{equation*}
which looks like the original version \eqref{AP-CSF, weak} with the Lagrangian interpolation introduced for the first term.
\end{remark}

\smallskip

Next we present the main results of this paper.

\smallskip

%\smallskip
\begin{theorem}({\textbf{Convexity-preserving}})\label{Preservation of convexity}
Suppose the initial curve $X_h(\xi, 0)=I_hX^0$ is a convex $N$-polygon, then it is always a convex $N$-polygon during the evolution by \eqref{Semidiscrete, weak} if $q_j>0$ for all $j$.
 %polygon $\cP^0$ is a convex $N$-polygon, $N\ge 5$, then the polygon $\cP(t)$ evolved by geometric flow (\ref{Semidiscrete, lumped mass})  is also a convex $N$-polygon.
\end{theorem}
\medskip

\begin{theorem}({\textbf{Perimeter-decreasing}})\label{Perimeter decreasing property}
Let $X_h$ be the solution of \eqref{Semidiscrete, weak} with convex  initial data, then the perimeter of the closed curve is decreasing, i.e.,
\begin{equation}
	\frac{\d }{\d t}L_h \le 0.
\end{equation}

\end{theorem}

\begin{remark}
For the cases of classical CSF or anisotropic CSF and the corresponding solutions based on similar formulations as in \eqref{Semidiscrete, weak}, by direct computations based on a finite difference structure, it was shown in \cite{Dziuk1999,Dziuk1994} that the length of each element of $X_h$ is decreasing, i.e., $q_j'(t)\le 0$ for $1\le j\le N$. This directly implies the perimeter-decreasing property. We point out  that this property can also be obtained by standard energy estimates for the CSF. However, in our AP-CSF case, both arguments fail to derive the perimeter-decreasing property. We have to carry out a more careful investigation in which the convexity property plays a vital role (Section 3.2).
\end{remark}

\medskip

\begin{theorem}({\textbf{Error estimate}})\label{Error estimate}
Let $X(\xi,t)$ be a solution of (\ref{AP-CSF, parametrization})  satisfying Assumption 2.2. Assume that the partition of $\S^{1}$ satisfies Assumption 2.1. Then there exists $h_0>0$ such that for all $0<h\le h_0$, there exists a unique semi-discrete solution $X_{h}$ for \eqref{Semidiscrete, weak}. Furthermore, the solution satisfies
	\begin{equation*}
	\begin{split}
		\int^T_0\|\p_t X-\p_t X_h\|^2_{L^2}\mathrm{d}s+\sup_{[0,T]}\|X-X_h\|^2_{H^1}\le Ch^2,
	\end{split}
	\end{equation*}
where $h_0$ and $C$ depend on $C_p, C_P,\kappa_1,\kappa_2, T$ and $K(X)$.
	 In particular, we have
\begin{equation}\label{Nondegenerate of edge length}
	\min_{j}q_j(t) >0,\quad \forall\ t\in [0,T].
\end{equation}
\end{theorem}

\section{Convexity-preserving and perimeter-decreasing properties}
Similar as in \cite{Deckelnick-Dziuk-Elliott}, we rewrite (\ref{Semidiscrete, weak}) into a lumped mass formulation. More precisely, taking
	\[	v_h=(\varphi_j,0)=\Big(\frac{\xi-\xi_{j-1}}{h_j}\chi_{I_j}+
\frac{\xi_{j+1}-\xi}{h_{j+1}}\chi_{I_{j+1}},0\Big),
		\]
		in (\ref{Semidiscrete, weak}), where $\chi$ is the characteristic function, calculations in \cite{Dziuk1999} give
		\begin{align*}
	&\quad\int_{\S^1}|\p_\xi X_h|\p_t X_h\cdot v_h\ \d \xi+ \int_{\S^1} \frac{\p_\xi X_h}{|\p_\xi X_h|} \cdot \p_\xi v_h\ \d \xi &+\int_{\S^1}\frac{\mathbf{h}^2|\p_\xi X_h|}{6}\p_\xi\p_t X_h\cdot \p_\xi v_h\mathrm{~d} \xi\\
	&=\frac{q_j+q_{j+1}}{2}\dot X_j^{[1]}-\big(\cT_{j+1}^{[1]}-\cT_{j}^{[1]}\big),
\end{align*}
		where $a^{[1]}$ denotes the first component of the vector $a\in \R^2$, and
\[\cT_{j}:=\frac{X_j-X_{j-1}}{|X_j-X_{j-1}|}, \quad \cN_j=\Big(\frac{X_j-X_{j-1}}{|X_j-X_{j-1}|}\Big)^{\perp}.\]
 For the last term involving the perimeter, we similarly compute
	\begin{align*}
			&\int_{\S^1}\frac{2\pi }{L_h} (\p_\xi X_h)^\perp\cdot v_h \  \d \xi\\
		=& \int_{\xi_{j-1}}^{\xi_j}\frac{2\pi }{L_h}\frac{q_j}{h_j} \cN_j\cdot\Big(\frac{\xi-\xi_{j-1}}{h_j},0\Big)\ \d\xi+\int_{\xi_{j}}^{\xi_{j+1}}\frac{2\pi }{L_h} \frac{q_{j+1}}{h_{j+1}} \cN_{j+1}\cdot \Big(\frac{\xi_{j+1}-\xi}{h_{j+1}},0\Big)\ \d\xi\\
		=& \frac{\pi}{L_h}q_j\cN_j^{[1]} +\frac{\pi}{L_h} q_{j+1}\cN_{j+1}^{[1]}.
		\end{align*}
Similarly taking $v_h=(0,\varphi_j)$ yields the equation for the second component. Thus the weak formulation \eqref{Semidiscrete, weak} is equivalent to the following lumped mass formulation
\begin{equation}\label{Semidiscrete, lumped mass}
	\frac{q_j+q_{j+1}}{2}\dot X_j=\cT_{j+1}-\cT_{j}-\frac{\pi}{L_h}(q_j\cN_j+q_{j+1}\cN_{j+1})=
\cT_{j+1}-\cT_{j}-\frac{\pi}{L_h}(X_{j+1}-X_{j-1})^\perp.
	\end{equation}
Hence it remains to solve the ODE system \eqref{Semidiscrete, lumped mass} and the image of $X_h$ is a polygon with $X_j(t)$ as the vertices.

For further studies, we derive some important formulae which will be used frequently.  Straightforward calculations as in \cite[Proposition 4.1]{Deckelnick-Dziuk-Elliott}, \cite[Lemma 3.1, Lemma 4.2]{Dziuk1999} and \cite[Lemma 2.4, Lemma 3.2]{Pozzi-Stinner} lead to
\begin{align}
		&\qquad\qquad\p_t|\p_\xi X | = -|\p_\xi X||\p_t X |^2+\p_t X\cdot R |\p_\xi X|,\label{Length element equation}\\
	\frac{\d }{\d t}q_j		&=-\frac{1}{q_j+q_{j+1}}|\cT_{j+1}-\cT_{j}|^2-\frac{1}{q_j+q_{j-1}}|\cT_{j-1}-\cT_{j}|^2 +\cT_j\cdot\l(R_j-R_{j-1}  \r) \label{dq1}\\
		&=-\frac{q_j+q_{j+1}}{4}|\dot X_j-R_j|^2-\frac{q_j+q_{j-1}}{4}|\dot X_{j-1}-R_{j-1}|^2+\cT_j\cdot\l(R_j-R_{j-1}  \r),\label{dq2}
			\end{align}
	where for simplicity we denote
	\be\label{rrj}
	R:=-\frac{2\pi}{L}\cN,\quad R_j:=-\frac{2\pi}{L_h}\frac{\cN_{j}q_j+\cN_{j+1}q_{j+1}}{q_j+q_{j+1}}.
	\ee
By using  above quantities, the equation (\ref{Semidiscrete, lumped mass}) can also be written as
\begin{equation}\label{Semidiscrete, another formulation}
	\dot X_j-R_j=2\l(\cT_{j+1}-\cT_j\r)/(q_j+q_{j+1}).
\end{equation}

In this section, we will prove that this semi-discrete geometric flow preserves the convexity of polygons under the nondegeneration property of vertices, which can be guaranteed by \eqref{Nondegenerate of edge length}. Furthermore, the perimeter-decreasing property is also shown for convex initial data.

\subsection{Proof of Theorem \ref{Preservation of convexity}}
First, we carry out some clarifications concerning with a polygon. We denote $\cP=(Y_1,\ldots,Y_N)$ as an $N$-polygon with $Y_j$ being its vertices and $\overline{Y_{j-1}Y_j}$ being the edge connecting $Y_{j-1}$ and $Y_j$. We emphasize that $\cP=(Y_1,\ldots,Y_N)$ has exactly $N$ sides, i.e., none of any three adjacent points are collinear. We say $\cP$ is a convex polygon if it is the boundary of a convex set. Without loss of generality, we assume that $Y_j$ is arranged in an anticlockwise way. We define the oriented area of  three points $Y_1,Y_2,Y_3\in \R^2$ as
\[
\mathrm{Area}(Y_1,Y_2,Y_3):=\frac{1}{2}\begin{vmatrix}
	1 & x_{1} & y_{1}\\
	1 & x_{2} & y_{2}\\
	1 & x_{3} & y_{3}\\
\end{vmatrix}=\frac{1}{2}(Y_3-Y_2)\cdot (Y_2-Y_1)^\perp,
\]
where $Y_i=(x_i,y_i)$, $i=1,2,3$. The following characterizations of convexity are straightforward.

\medskip
\begin{lemma}\label{Convexity lemma}
Let $\cP=(X_1,\ldots,X_N)$ be an $N$-polygon. The following statements are equivalent:
\begin{itemize}
\item [(i)] The $N$-polygon $\cP=(X_1,\ldots,X_N)$ is convex;\vspace{2mm}

 \item [(ii)] Any internal angle $\angle X_{j-1}X_jX_{j+1}<\pi$ for $j=1,\ldots,N$;\vspace{2mm}

 \item [(iii)] $S_j:=\mathrm{Area}(X_{j-1},X_j,X_{j+1})>0$ for  $j=1,\ldots,N$;\vspace{2mm}

     \item [(iv)] $S_j^k:=\mathrm{Area}(X_{j-1},X_j,X_{k})>0$ for  $j=1,\ldots,N$ and $k\neq j-1,j$.
     \end{itemize}
 Here, we set $X_0=X_N$ and $X_{N+1}=X_1$ when involved.
\end{lemma}

\begin{proof}
Clearly we have (i)$\Leftrightarrow$(ii)$\Leftrightarrow$(iii) and (iv)$\Rightarrow$(iii). It suffices to show (i)$\Rightarrow$(iv). Indeed, by the support property of convex polygons \cite[Theorem 4.2]{Gruber}, for any $X\in \overline{X_{j-1}X_j}$, there exists a support hyperplane $\ell(X)$ such that $\mathcal{P}$ is contained in one of the two closed halfspaces determined by $\ell(X)$. It's obvious that $\overline{X_{j-1}X_j}\subseteq\ell(X)$. In particular, for any $k\neq j-1,j$, $X_k$ lies in the same halfspace determined by $\ell(X)$, i.e., there exists a nonzero vector $\cN\in \R^2$ such that
\begin{equation}\label{Convexity condition 1}
	\cN\cdot (X_j-X_{j-1})=0,\quad  \cN\cdot (X_{k}-X_{j-1})>0 \,\,(k\neq j-1, j).
\end{equation}
Thus we can write $\cN=\varepsilon(X_j-X_{j-1})^{\perp}$. Noticing
\[2S_j^k
=(X_k-X_j)\cdot (X_j-X_{j-1})^\perp=(X_k-X_{j-1})\cdot \cN/\varepsilon,\quad k\neq j-1, j,
\]
which together with (iii) implies $\varepsilon>0$, this yields (iv) by recalling \eqref{Convexity condition 1} and the proof is completed.
\end{proof}

\smallskip

Inspired by \eqref{Xh} and \eqref{Semidiscrete, lumped mass}, to prove Theorem \ref{Preservation of convexity}, it suffices to show that for any $t>0$, the $N$-polygon $\cP=(X_1,\ldots,X_N)$ is convex. We first compute the evolution formula of the oriented area of the triangles consisting of three adjacent vertices.

\begin{lemma}
Under the flow \eqref{Semidiscrete, lumped mass}, if $q_j>0$ for any $j$, then the oriented area $S_j(t)$ satisfies
\begin{equation}\label{DSt}
\begin{split}
\frac{\d}{\d t}S_j(t)
	&= -a_j\cdot S_j+b_j\cdot S_{j+1}^{j-2}+c_j\cdot S_j^{j+2}
+\frac{1}{q_{j}+q_{j+1}}\frac{\pi}{L_h}|X_{j+1}-X_{j-1}|^2\\
&\quad+\frac{1}{q_{j-1}+q_{j}}\frac{\pi}{L_h}(X_j-X_{j+1})\cdot(X_j-X_{j-2})\\
&\quad+\frac{1}{q_{j+1}+q_{j+2}}\frac{\pi}{L_h}(X_j-X_{j-1})
\cdot(X_{j}-X_{j+2}),
\end{split}	
\end{equation}
where $a_j,b_j,c_j$ are positive functions defined by
\begin{align*}
	a_{j}&=\frac{2}{q_j q_{j-1}}+\frac{2}{q_{j}q_{j+1}}+\frac{2}{q_{j+1}q_{j+2}},\quad
	b_j=\frac{2}{q_{j-1}(q_{j-1}+q_j)},\quad  c_j=\frac{2}{q_{j+2}(q_{j+1}+q_{j+2})}.
\end{align*}

\end{lemma}

\begin{proof}
By definition, it can be observed
\[S_j=\frac{1}{2}(X_{j+1}-X_j)\cdot
(X_j-X_{j-1})^\perp=\frac{q_j q_{j+1}}{2}\mathcal{T}_{j+1}\cdot \mathcal{N}_j.\]
Employing the flow equation \eqref{Semidiscrete, lumped mass}, we derive
\begin{align*}
\frac{\d}{\d t}S_j&	=\frac{1}{2}\frac{\d}{\d t}\l(X_{j+1}\cdot X_j^\perp+X_{j}\cdot X_{j-1}^\perp-X_{j+1}\cdot X_{j-1}^\perp \r)\\
	&=\frac{1}{2}\frac{\d X_{j+1}}{\d t}\cdot(X_{j}^\perp-X_{j-1}^\perp )-\frac{1}{2}\frac{\d X_{j}}{\d t}\cdot(X_{j+1}^\perp-X_{j-1}^\perp ) +\frac{1}{2}\frac{\d X_{j-1}}{\d t}\cdot(X_{j+1}^\perp-X_{j}^\perp)\\
	&\triangleq  :J_1+J_2+J_3,
	\end{align*}
where we have used the property $u\cdot v^\perp=-v\cdot u^\perp$ for any $u ,v\in \mathbb{R}^2$. Applying \eqref{Semidiscrete, lumped mass}, one can calculate $J_j$ as
\begin{align*}
&\quad\l(q_{j+1}+q_{j+2}\r)J_1=
\frac{q_j}{2}(q_{j+1}+q_{j+2})\dot X_{j+1}\cdot\mathcal{N}_j\\	&=q_j\Big(\mathcal{T}_{j+2}-\mathcal{T}_{j+1}-\frac{\pi}{L_h}
\big(X_{j+2}-X_j\big)^\perp\Big)\cdot \mathcal{N}_j\\
&=\frac{X_{j+2}-X_j}{q_{j+2}}\cdot q_j \mathcal{N}_j-\Big(1+\frac{q_{j+1}}{q_{j+2}}\Big)\mathcal{T}_{j+1}\cdot
q_j\mathcal{N}_j-\frac{\pi}{L_h}(X_{j+2}-X_{j})\cdot q_j\mathcal{T}_j\\
&=\frac{2}{q_{j+2}}S_j^{j+2}-\Big(\frac{2}
{q_{j+1}}+\frac{2}{q_{j+2}}\Big)S_j-\frac{\pi}{L_h}(X_j-X_{j-1})\cdot(X_{j+2}-X_{j}).
\end{align*}
Similarly one easily gets
\begin{align*}
&\l(q_{j}+q_{j+1}\r)J_2
=-\Big(\frac{2}{q_j}+\frac{2}{q_{j+1}}\Big)S_j+\frac{\pi}{L_h} \l|X_{j+1}-X_{j-1}\r|^2,\\
&\l(q_{j-1}+q_j\r)J_3 =-\Big(\frac{2}{q_{j-1}}+\frac{2}{q_j}\Big)S_j+\frac{2}{q_{j-1}}S_{j+1}^{j-2}	+\frac{\pi}{L_h}(X_j-X_{j+1})\cdot(X_j-X_{j-2}).
\end{align*}
Combining the above equations together yields \eqref{DSt} immediately.
\end{proof}

\smallskip

Now we turn to the proof of convexity-preservation.

\emph{Proof of Theorem \ref{Preservation of convexity}.}
Define the function $F(t):=\min\{S_j(t),\,\, 1\le j\le N\}$. It follows from the assumption that $F(0)>0$. By the definition of $N$-polygon and Lemma \ref{Convexity lemma} (iii), it suffices to show $F(t)>0$, for $t\in (0,T]$.

We argue by contradiction. Suppose the contrary, by continuity  there exists the smallest time $0<t_0\le T$ such that $F(t_0)= 0$. Then we have
\begin{enumerate}
	\item[(1)] there exists some triangle such that the oriented area achieves zero at $t_0$, without loss of generality, we may assume $S_2(t_0)=\mathrm{Area}(X_1,X_2,X_3)\ (t_0)=0$;

\item[(2)] the $N$-polygon $\mathcal{P}(t)=(X_1(t), \ldots, X_N(t))$ is convex for $0\le t<t_0$, hence by Lemma \ref{Convexity lemma} (iv), it holds
\[
S_j^k(t)>0,\quad  \forall\ 0\le t<t_0,\quad \forall\ j=1,\ldots,N,\ \forall\ k\neq j-1,j.\]
\end{enumerate}
Thus by \eqref{DSt}, one has
\be\label{ds2}
0\ge \frac{\d}{\d t}S_2(t_0)=b_2\cdot S_{3}^0(t_0)+c_2\cdot S_2^{4}(t_0)+Q(t_0),
\ee
where
\begin{align*}
Q(t_0)&=\frac{1}{q_{2}+q_{3}}\frac{\pi}{L_h}|X_{3}-X_{1}|^2 (t_0)+\frac{1}{q_{1}+q_{2}}\frac{\pi}{L_h}(X_2-X_{3})\cdot(X_2-X_{0})\ (t_0)\\
	&\quad +\frac{1}{q_{3}+q_{4}}\frac{\pi}{L_h}(X_2-X_{1})\cdot(X_{2}-X_{4})\ (t_0).
	\end{align*}
Noticing that (1) implies that $X_1(t_0), X_2(t_0), X_3(t_0)$ are collinear. There are two possibilities: (i) $(X_2-X_1)\cdot (X_3-X_2)(t_0)>0$; (ii) $(X_2-X_1)\cdot (X_3-X_2)(t_0)<0$. Next we discuss it case by case.

\noindent\textbf{Case (i): $(X_2-X_1)\cdot (X_3-X_2)(t_0)>0$.}

Firstly by (2) and continuity, we easily find that $S_3^0(t_0)\ge 0$, $S_2^4(t_0)\ge 0$.
We claim that $Q(t_0)\ge 0$. Actually, notice that in this case it holds
\[|X_3-X_1|(t_0)=|X_3-X_2|(t_0)+|X_2-X_1|(t_0)=q_2(t_0)+q_3(t_0),\]
 which implies
	\begin{align*}
\frac{\pi}{L_h}\frac{|X_{3}-X_{1}|^2}{q_{2}+q_{3}}\ (t_0)=\frac{\pi}{L_h}\frac{(q_2+q_3)^2}{q_{2}+q_{3}}\ (t_0)=\frac{\pi}{L_h}\l(q_2+q_3\r)\ (t_0).
	\end{align*}
On the other hand, by the triangle inequality, one can estimate
	\begin{align*}
\frac{1}{q_{1}+q_{2}}\frac{\pi}{L_h}(X_2-X_{3})\cdot(X_2-X_{0})
	&\ge -\frac{1}{q_{1}+q_{2}}\frac{\pi}{L_h}\l(q_3\cdot (q_1+q_2)  \r)= -\frac{\pi}{L_h}q_3,\\
\frac{1}{q_{3}+q_{4}}\frac{\pi}{L_h}(X_2-X_{1})\cdot(X_{2}-X_{4})
		&\ge  -\frac{\pi}{L_h}q_2.
	\end{align*}
Thus
\[Q(t_0)\ge
		\frac{\pi}{L_h}\l(q_2+q_3\r)(t_0)-\frac{\pi}{L_h}q_3(t_0)-
\frac{\pi}{L_h}q_2(t_0)=0.
\]
Recalling \eqref{ds2}, all above inequalities become equalities, i.e., $S_{3}^{0}(t_0)=S_2^4(t_0)=0$, and
\[|X_2-X_0|(t_0)=q_1(t_0)+q_2(t_0),\quad |X_2-X_4|(t_0)=q_2(t_0)+q_3(t_0).\]
This means $X_0,X_1,X_2,X_3,X_4$ are collinear at $t_0$ and are arranged in order, i.e.,
$(X_{j+1}-X_j)(t_0)=d_j(X_1-X_0)(t_0)$ with $d_j>0$, $j=1,2,3$.  In particular,
\[S_3(t_0)=0,\quad (X_3-X_2)\cdot (X_4-X_3)(t_0)>0.\]
Repeating the above procedure by another $N-4$ times, we get that at time $t_0$, all vertices $(X_0,X_1,\ldots,X_N)$ are collinear and are arranged in order, i.e.,
\[X_{j+1}(t_0)-X_j(t_0)=d_j(X_1(t_0)-X_0(t_0)),\quad d_j>0,\quad j=1, 2, \ldots, N-1,\]
which contradicts with the periodic condition $X_0=X_N$.
\medskip

\noindent\textbf{Case (ii): $(X_2-X_1)\cdot (X_3-X_2)(t_0)<0$.}

This means $(X_3-X_2)(t_0)=d_2(X_2-X_1)(t_0)$ with $d_2<0$. By (2) and continuity, we have
$S_3(t_0)\ge 0$, $S_2^4(t_0)\ge 0$. On the other hand, by definition, one finds
\begin{align*}
S_3(t_0)&=\frac{1}{2}(X_4-X_3)\cdot (X_3-X_2)^\perp(t_0)=\frac{1}{2}(X_4-X_2)\cdot (X_3-X_2)^\perp(t_0)\\
&=\frac{d_2}{2}(X_4-X_2)\cdot (X_2-X_1)^\perp(t_0)=d_2 S_2^4(t_0)\le 0.
\end{align*}
It follows $S_3(t_0)=0$ and $X_1, X_2, X_3, X_4$ are collinear.
We claim that
\[(X_4-X_3)\cdot (X_3-X_2)(t_0)<0.\]
 Otherwise, Case (i) happens for the collinear points $X_2, X_3, X_4$. Differentiating $S_3$ at $t_0$ and repeating the arguments as in Case (i) involving $S_3$ will lead to the conclusion that $X_1, X_2, X_3, X_4, X_5$ are collinear and are arranged in order, which contradicts with the premise that $(X_2-X_1)\cdot (X_3-X_2)(t_0)<0$. Thus it holds
 \[S_3(t_0)=0,\quad (X_4-X_3)\cdot (X_3-X_2)(t_0)<0.\]
 Repeating this argument, we can conclude that all vertices $X_1,X_2,\ldots,X_N$ are collinear, furthermore, every three adjacent vertices are interlaced, i.e.,
 \[
 (X_{j+1}-X_j)\cdot (X_j-X_{j-1})(t_0)<0,\quad j=1,\ldots,N.
 \]
In particular, all exterior angles of the polygon $\cP(t_0)$ are $\pi$. On the other hand, noticing each exterior angle $\alpha_j$ is continuous, by continuity and convexity, we have
\[
N\pi=\sum_{j=1}^N\alpha_j(t_0)=\lim_{t\rightarrow t_0}\sum_{j=1}^N\alpha_j(t)=2\pi,
\]
which leads to a contradiction since $N\ge 3$.
$\square\hfill$

\begin{remark}
 A similar argument holds for Dziuk's semi-discrete scheme \cite{Dziuk1994} for the CSF. More precisely, under nondegeneration of vertices ($q_j>0$) we can first  prove that if the initial polygon is convex, then the evolved polygon under the semi-discrete scheme  of the  CSF is also convex unless all vertices are collinear at some $t_0>0$, in which case the area vanishes, i.e., Area$(\cP(t_0))=0$. On the other hand, applying the error estimate of the scheme for the CSF \cite{Dziuk1994,Dziuk1999}, we arrive at
 \begin{align*}
 	\l|\mathrm{Area}(\Gamma_{t_0})-\mathrm{Area}(\cP(t_0))\r|
 	&=\Big|\int_{\S^1}\p_\xi x\cdot y\ \d \xi -\int_{\S^1}\p_\xi x_h\cdot y_h\ \d \xi  \Big|\\
 	&\le \int_{\S^1}|\p_\xi x-\p_\xi x_h |\cdot |y|\ \d \xi+ \int_{\S^1} |\p_\xi x_h|\cdot |y-y_h|\ \d \xi\\
 	&\le C\sup_{[0,T]}\|X-X_h\|_{H^1(\S^1)}\le Ch,
 \end{align*}
 where $\Gamma_{t}$ represents the real curve driven by the AP-CSF.
 This implies that $\mathrm{Area}(\cP(t_0))$ stays away from zero, if $\mathrm{Area}(\Gamma_t)$ has a positive lower bound for $t\in (0,T]$ and $h$ is small enough. This leads to a contradiction!
\end{remark}

\subsection{Proof of Theorem \ref{Perimeter decreasing property}}

\

\smallskip
\begin{proof}
Applying \eqref{dq1}, we get the derivative of the perimeter
\begin{align*}
	\frac{\d }{\d t}L_h
	&=\sum_{j=1}^N\frac{\d }{\d t}q_j =\sum_{j=1}^N\Big(-\frac{1}{q_j+q_{j+1}}|\cT_{j+1}-\cT_{j}|^2-\frac{1}{q_j+q_{j-1}}|\cT_{j-1}-\cT_{j}|^2 \Big)\\
	&\quad + \frac{2\pi}{L_h}\sum_{j=1}^N\Big(-\frac{q_{j+1}}{q_j+q_{j+1}} \cT_j\cdot \cN_{j+1}+\frac{q_{j-1}}{q_j+q_{j-1}}\cT_j\cdot \cN_{j-1} \Big)\\
&=-2\sum_{j=1}^N \frac{|\cT_{j+1}-\cT_{j}|^2}{q_j+q_{j+1}} -\frac{2\pi}{L_h}\sum_{j=1}^N\cT_j\cdot \cN_{j+1},
\end{align*}
where we have used the fact that
\begin{align*}
		&\sum_{j=1}^N\Big(-\frac{q_{j+1}}{q_j+q_{j+1}} \cT_j\cdot \cN_{j+1}+\frac{q_{j-1}}{q_j+q_{j-1}}\cT_j\cdot \cN_{j-1} \Big)\\
		%=&\ \sum_{j=1}^N\Big(-\frac{q_{j+1}}{q_j+q_{j+1}} \cT_j\cdot \cN_{j+1}+\frac{q_j}{q_j+q_{j+1}} \cT_{j+1}\cdot \cN_{j} \Big)\\
		=&\ \sum_{j=1}^N\Big(-\frac{q_{j+1}}{q_j+q_{j+1}} \cT_j\cdot \cN_{j+1}-\frac{q_j}{q_j+q_{j+1}} \cN_{j+1}\cdot \cT_{j} \Big)= -\sum_{j=1}^N\cT_j\cdot \cN_{j+1}.
	\end{align*}
We denote $\alpha_j$ by the exterior angle of the polygon at $X_j$. By Theorem \ref{Preservation of convexity}, $\cP(t)$ keeps convex for all $t$, which implies $0<\alpha_j<\pi$ for $j=1,\ldots,N$, and $\sum\limits_{j=1}^N \alpha_j=2\pi$.
Direct computations yield
\begin{align*}
	\l|\cT_{j+1}-\cT_j \r|^2=2-2\cos \alpha_j =4\sin^2(\alpha_j/2),\quad \text{and}\quad \cT_j\cdot \cN_{j+1}=-\sin \alpha_j.
\end{align*}
By Cauchy-Schwarz inequality, one easily gets
\begin{align*}
	\sum_{j=1}^N2\sin\l( \frac{\alpha_j }{2}\r)
	&\le \Big(\sum_{j=1}^N\frac{4\sin^2\l(\frac{\alpha_j }{2}\r)}{q_j+q_{j+1}} \Big)^{\frac12} \Big(\sum\limits_{j=1}^N q_j+q_{j+1} \Big)^{\frac12}= \l(2L_h\r)^{\frac12} \Big(\sum_{j=1}^N\frac{4\sin^2\l( \frac{\alpha_j }{2} \r)}{q_j+q_{j+1}} \Big)^{\frac12}.
\end{align*}
Hence we derive
\begin{align*}
	\frac{\d L_h}{\d t}	
	&=-2\sum\limits_{j=1}^N\Big(\frac{4\sin^2\big(\frac{\alpha_j }{2} \big)}{q_j+q_{j+1}}-\frac{\pi}{L_h}\sin\alpha_j  \Big)
	%&\le -2\l(\frac{\l(\sum\limits_{j=1}^N2\sin\l( \frac{\alpha_j }{2}\r)\r)^{2}}{2L_h}-\frac{\pi}{L_h}\sum_{j=1}^N\sin\alpha_j\r)\\
\le\frac{-2}{L_h}\Big(2 \big(\sum_{j=1}^N\sin\big( \frac{\alpha_j }{2}\big)\big)^{2}-\pi \sum_{j=1}^N\sin\alpha_j\Big)\le 0,
\end{align*}
where in the last inequality we have utilized  a trigonometric inequality (cf. Lemma \ref{A trigonometric  lemma} below) and the proof is completed.
\end{proof}

\smallskip

\begin{lemma}\label{A trigonometric  lemma}
Define
\[
f_N(\beta_1,\ldots,\beta_N):=\Big(\sum_{j=1}^N\sin \beta_j  \Big)^2-\frac{1}{2}\Big(\sum\limits_{j=1}^N\beta_j\Big) \Big(\sum\limits_{j=1}^N\sin(2\beta_j)\Big),\quad 0\le \beta_j\le \frac{\pi}{2}.
\]
Then it holds $f_N(\beta_1,\ldots,\beta_N)\ge 0$.
	
\end{lemma}

\begin{proof}
	We prove $f_N(\beta_1,\ldots,\beta_N)\ge 0$ by induction. For $N=1$, we have
\[
	f_1(\beta_1)=\sin^2 \beta_1-\frac{\beta_1}{2}\cdot \sin(2\beta_1)=\sin\beta_1\cos\beta_1\l(\tan\beta_1-\beta_1 \r)\ge 0.
\]
Now suppose $f_{N-1}(\beta_1,\ldots,\beta_{N-1})\ge 0$, we first compute
\begin{align*}
	\frac{\p f_N (\beta_1,\ldots,\beta_N)}{\p \beta_N}
	&=2\cos\beta_N\cdot \sum_{j=1}^N\sin \beta_j-\frac{1}{2}\sum_{j=1}^N\sin(2\beta_j)-\sum_{j=1}^N\beta_j\cdot\cos(2\beta_N)\\
	&\ge 2\cos^2\beta_N\cdot \sum_{j=1}^N\sin \beta_j-\sum_{j=1}^N\sin\beta_j\cos\beta_j-\sum_{j=1}^N\beta_j\cdot\cos(2\beta_N)\\
	%&=(1+\cos(2\beta_N))\cdot \sum_{j=1}^N\sin \beta_j-\sum_{j=1}^N\sin\beta_j\cos\beta_j-\sum_{j=1}^N\beta_j\cdot\cos(2\beta_N)\\
	&= \sum_{j=1}^N\l(\sin \beta_j-\sin\beta_j\cos\beta_j-\l(\beta_j-\sin\beta_j \r)\cdot\cos(2\beta_N)\r)\\
	&\ge \sum_{j=1}^N\l(\sin \beta_j-\sin\beta_j\cos\beta_j-\l(\beta_j-\sin\beta_j \r)\r)=:\sum_{j=1}^N B_j(\beta_j).
\end{align*}
Noticing
\[
	\frac{\p B_j}{\p \beta_j}=\cos\beta_j-\cos^2\beta_j+\sin^2\beta_j-(1-\cos\beta_j)
=2\cos\beta_j-2\cos^2\beta_j\ge  0,\]
this implies $B_j$ is increasing and particularly,
\[
	B_j(\beta_j)\ge B_j(0)=0,\quad \beta_j\in [0, \pi/2],\quad j=1,\ldots,N.
\]
Hence one gets $\frac{\p f_N (\beta_1,\ldots,\beta_N)}{\p \beta_N}\ge 0$, and by induction,
\[
	f_N(\beta_1,\ldots,\beta_N)\ge f_N(\beta_1,\ldots,\beta_{N-1},0)=f_{N-1}
(\beta_1,\ldots,\beta_{N-1})\ge  0,\]
which completes the proof.
\end{proof}

\section{Proof of Theorem \ref{Error estimate}}
In this section we present the error estimate by following the lines of Dziuk's argument \cite{Dziuk1999} and Pozzi-Stinner's computation \cite{Pozzi-Stinner}. We establish the stability  estimate and length element difference under the assumption of boundedness of the semi-discrete length element. Then a bound of the semi-discrete length element is given. All above preliminary estimates together with the continuity argument enable us to derive
the desired error bound. Throughout this section, we suppose  Assumptions 2.1 and 2.2 are always valid and we denote $C>0$ by a general constant and may change from line to line. For simplicity we omit the space whenever the norm is defined on $\S^1$.

We first give the stability  estimate.
\begin{lemma}\label{Consistency estimate lemma}
Suppose further the solution of \eqref{Semidiscrete, weak} satisfies
  \begin{equation}\label{Control assumption, consistency estimate}
  	\inf _{\xi}\left|\partial_{\xi} X_{h}\right|\ge c_0>0,\quad \text{and}\quad \sup_{\xi}\left|\partial_{\xi} X_{h}\right|\le C_0,\quad \forall\ 0\le t\le T^*\le T.
  \end{equation}
 Then for any $t\in [0,T^*]$, we have
  \begin{equation}\label{cons}
  \int_{0}^{t} \int_{\S^{1}}|\p_t X-\p_t X_h|^{2}q_h \d \xi \d s +\sup\limits_{0\le s\le t}\int_{\S^{1}}|\cT-\cT_h|^{2}q_h\d \xi\leq C \int_{0}^{t} \|q-q_h\|_{L^2}^{2}\d s +C h^{2},
    \end{equation}
where $\cT=\frac{\partial_{\xi} X}{\left|\partial_{\xi} X\right|},\cT_h=\frac{\partial_{\xi} X_{h}}{\left|\partial_{\xi} X_{h}\right|}$, $q= |\p_\xi X|$, $q_h=|\p_\xi X_h |$ and $C$ depends on $C_p$, $C_P$, $\kappa_1$, $c_0$, $C_0$ and $K(X)$.
	
\end{lemma}
\smallskip
\begin{proof}
	We first notice that the boundedness of the length element will imply the boundedness of the perimeter. Indeed, by Assumption 2.2 and \eqref{Control assumption, consistency estimate}, one easily gets
 \begin{equation}\label{Control of perimeter}
 	2\pi \kappa_1\le  L\le 2\pi \kappa_2 ,\quad 2\pi c_0\le L_h\le 2\pi C_0,\quad \forall\ t\in [0,T^*].
 \end{equation}
  Recalling $\p_\xi X\neq 0$, $\p_\xi X_h\neq 0$, $|\mathcal{T}|=|\mathcal{T}_h|=1$, this enables us to write the following
\begin{equation}\label{Basic equality}
 	\begin{split}
 		 |\partial_\xi X-\partial_\xi X_h |^2
                &=|\partial_\xi X|^2+|\partial_\xi X_h|^2-2\partial_\xi X\cdot \partial_\xi X_h\\
                &=(q-q_h)^2+2qq_h-2qq_h\mathcal{T}\cdot\mathcal{T}_h\\
                &=(q-q_h)^2+qq_h(2-2\mathcal{T}\cdot\mathcal{T}_h)=
               (q-q_h)^2+qq_h |\mathcal{T}-\mathcal{T}_h|^2.
 	\end{split}
 \end{equation}
  Taking the difference between \eqref{AP-CSF, weak} and \eqref{Semidiscrete, weak}, we obtain the error equation
  \begin{align*}
  	&\int_{\S^1}\l( \left|\partial_{\xi} X\right| \p_t X -\left|\partial_{\xi} X_{h}\right| \p_t X_{h} \r)\cdot v_h  \mathrm{~d} \xi+\int_{\S^1} \l(\cT-\cT_h  \r)\cdot \partial_{\xi} v_h \mathrm{~d} \xi\\
  	&+\int_{\S^1}\Big(\frac{2\pi}{L}\l(\p_\xi X\r)^{\perp}-\frac{2\pi}{L_h}\l(\p_\xi X_h\r)^{\perp} \Big)\cdot v_h  \mathrm{~d} \xi
  =\int_{\S^1}\frac{\mathbf{h}^2|\p_\xi X_h|}{6}\p_\xi\p_t X_h\cdot \p_\xi v_h\mathrm{~d} \xi
  \end{align*}
 holds for any $v_h\in V_h$. Taking $v_h=I_h(\p_t X)-\p_t X_{h} \in V_h$ in the above equation yields
  \begin{align*}
  	&\quad\int_{\S^1}|\p_t X-\p_t X_{h} |^2q_h \d\xi +\int_{\S^1}\l(\cT-\cT_h  \r)\l(\p_\xi\p_t X- \p_\xi\p_t X_h\r) \d\xi\\
  	&=\int_{\S^1}\p_t X\cdot (q_h-q)\l(I_h\p_t X-\p_t X_{h} \r)\d\xi+\int_{\S^1}\frac{\mathbf{h}^2q_h}{6}\p_\xi\p_t X_h\cdot \p_\xi \l(I_h\p_t X-\p_t X_{h} \r)\d\xi \\
  	&\quad +\int_{\S^1}q_h\cdot(\p_t X-\p_t X_h)(\p_t X-I_h\p_t X) d\xi+\int_{\S^1}\l(\cT-\cT_h  \r)\cdot \l(\p_\xi\p_t X- \p_\xi I_h\p_t X\r)\d \xi \\
  	&\quad +\int_{\S^1}\frac{2\pi}{L}\l(\p_\xi X-\p_\xi X_h \r)^\perp\cdot \l(\p_t X_h-I_h\p_t X\r)\d\xi\\
  	&\quad + \int_{\S^1}\Big(\frac{2\pi}{L}-\frac{2\pi}{L_h} \Big)\l(\p_\xi X_h \r)^\perp \cdot \l(\p_t X_h-I_h\p_t X\r)\d \xi 	\triangleq:\ J_1+J_2+J_3+J_4+J_5+J_6.
  \end{align*}
  The estimates of the second term on the left side and $J_j$ for $1\le j\le 4$ can be found in \cite[Lemma 5.1]{Dziuk1999}, which read as
\begin{align*}
  &\quad\int_{\S^1}\l(\cT-\cT_h \r)\cdot\l(\p_\xi\p_t X- \p_\xi\p_t X_h\r)  \d\xi\\
&\ge \frac{\d}{\d t}\Big(\int_{\S^1}\l(1-\cT\cdot\cT_h\r)q_h  \d\xi\Big)-C\|\p_\xi\p_tX\|_{L^\infty}\Big(\int_{\S^1}|\cT-\cT_h|^2 q_h d\xi
+\|q-q_h\|_{L^2}^2\Big),\\
  &J_1\le \varepsilon \int_{\mathbb{S}^1}|\partial_t X-\partial_tX_h|^2 q_h \mathrm{d}\xi +C(\varepsilon)\|\partial_t X\|^2_{L^\infty}\int_{\mathbb{S}^1}\frac{(q-q_h)^2}{q_h}\mathrm{d}\xi+C\|q_h\|_{L^\infty}^2\|\partial_t X\|_{H^1}^2h^2\\
  &\quad \le \varepsilon\int_{\mathbb{S}^1}\left|\partial_t X-\partial_t X_{h}\right|^{2}q_h \mathrm{d} \xi+C(\varepsilon)\left\|\partial_t X\right\|_{L^{\infty}}^{2} \|q-q_h\|_{L^2}^{2} +Ch^2\|\partial_tX\|_{H^1}^2,\\
  &J_2\le \frac{1}{24}\|q_h\|_{L^\infty}\|\partial_t X\|^2_{H^1}h^2\le Ch^2\|\partial_t X\|^2_{H^1},\\
 &J_3 \le  \varepsilon \int_{\mathbb{S}^1}|\partial_t X-\partial_t X_h|^2q_h \mathrm{d} \xi+C(\varepsilon)\|q_h\|_{L^\infty} \|\partial_tX\|_{H^1}^2h^2\\
  &\quad  \le \varepsilon \int_{\mathbb{S}^1}|\partial_t X-\partial_t X_h|^2q_h \mathrm{d} \xi+C(\varepsilon)h^2 \|\partial_tX\|_{H^1}^2,\\
&J_4\le C\|\partial_t X\|_{H^2}\|\mathcal{T}-\mathcal{T}_h\|_{L^2}h\le C\int_{\mathbb{S}^1}|\mathcal{T}-\mathcal{T}_h|^2 q_h\mathrm{d}\xi+Ch^2\|\partial_t X\|_{H^2}^2,
  	\end{align*}
where $\varepsilon$ is a generic small positive constant which will be chosen later. It remains to estimate $J_5$ and $J_6$. For $J_5$, we decompose it as $J_5=J_{51}+J_{52}$ with
\begin{align*}
  	J_{51}
		&=\int_{\S^1}\frac{2\pi}{L}\l(\p_\xi X-\p_\xi X_h \r)^\perp\cdot \l(\p_t X-I_h\p_t X\r) \d\xi,\\
		J_{52}&=\int_{\S^1}\frac{2\pi}{L}\l(\p_\xi X-\p_\xi X_h \r)^\perp\cdot \l(\p_t X_h-\p_t X\r)  \d \xi.
  \end{align*}
Applying Assumption 2.2, \eqref{Control of perimeter}, \eqref{Basic equality} and the interpolation estimate \eqref{Interpolation estimate}, we derive
  \begin{align*}
		J_{51}&
		\le  C\int_{\S^1}\l| \p_\xi X -\p_\xi X_h\r|^2\ \d \xi +C \int_{\S^1}|\p_t X-I_h\p_t X|^2 \d \xi \\
		&= C\Big(\int_{\S^1}q q_h\l(\l|\cT-\cT_h \r|^2 +(q-q_h)^2 \r)\d \xi+\l\|\p_t X-I_h\p_t X\r\|_{L^2}^2\Big) \\
		&\le C\l\|\p_\xi X\r\|_{L^\infty}\int_{\S^1}\l|\cT-\cT_h \r|^2q_h \d \xi + C\int_{\S^1}(q-q_h)^2 \d \xi+Ch^2\|\p_t X\|_{H^1}^2,\\
		J_{52}
		&\le C(\varepsilon ) \int_{\S^1}\l| \p_\xi X -\p_\xi X_h\r|^2 \d \xi +\varepsilon  \int_{\S^1}\l|\p_t X_h-\p_t X \r|^2q_h\d \xi\\
		&= C(\varepsilon ) \int_{\mathbb{S}^1}qq_h|\mathcal{T}-\mathcal{T}_h|^2 +(q-q_h)^2 \mathrm{d}\xi+\varepsilon  \int_{\S^1}\l|\p_t X_h-\p_t X \r|^2q_h\d \xi\\
		&\le C(\varepsilon )\l\|\p_\xi X\r\|_{L^\infty}\int_{\S^1}\l|\cT-\cT_h \r|^2q_h \d \xi + C(\varepsilon )\|q-q_h\|_{L^2}^2+\varepsilon  \int_{\S^1}\l|\p_t X_h-\p_t X \r|^2q_h\d \xi.
	\end{align*}
Similarly we decompose $J_6=J_{61}+J_{62}$ with
\begin{align*}
		J_{61}
		&=\int_{\S^1}\Big(\frac{2\pi}{L}-\frac{2\pi}{L_h} \Big)\l(\p_\xi X_h \r)^\perp \cdot \l(\p_t X-I_h\p_t X\r)\d \xi, \\
	J_{62}&=\int_{\S^1}\Big(\frac{2\pi}{L}-\frac{2\pi}{L_h} \Big)\l(\p_\xi X_h \r)^\perp \cdot \l(\p_t X_h-\p_t X\r)\d \xi.
	\end{align*}
Noticing
\begin{equation}\label{Estimate of perimeter}
\l|L-L_h\r|\le \|q-q_h\|_{L^1} \le C\|q-q_h\|_{L^2},
\end{equation}
this together with \eqref{Control of perimeter}, \eqref{Control assumption, consistency estimate} and \eqref{Interpolation estimate} lead to
\begin{align*}
		J_{61}
		&\le C|L-L_h|^2+ C\|\p_t X-I_h\p_t X\|_{L^2}^2\le C \|q-q_h\|_{L^2}^2+Ch^2\|\p_t X\|_{H^1}^2,\\
		J_{62}
		&\le C\int_{\S^1}|L_h-L|q_h^{\frac12}|\p_t X_h-\p_t X|\d\xi\le C(\varepsilon )|L_h-L|^2+\varepsilon \int_{\S^1}q_h\l|\p_t X_h-\p_t X \r|^2\ \d \xi \\
  &\le C(\varepsilon)\|q-q_h\|_{L^2}^2+\varepsilon \int_{\S^1}q_h|\p_t X_h-\p_t X|^2\d \xi.
	\end{align*}
Combining the above inequalities, we obtain
  \begin{align*}
  	&\int_{\S^1}|\p_t X-\p_t X_{h} |^2q_h \d\xi +\frac{\d}{\d t}\int_{\S^1}\l(1-\cT\cdot\cT_h\r)q_h \d\xi
  	\le 4\eps \int_{\S^1}|\p_t X-\p_t X_{h} |^2q_h \d\xi\\
  &\qquad\quad+C(\eps)h^2\|\p_t X\|^2_{H^2}+C(\eps, K(X)) \|q-q_h\|_{L^2}^2+C(\eps, K(X)) \int_{\S^1}|\cT-\cT_h|^2q_h\d\xi.
  \end{align*}
 Choosing $\eps$ small enough, integrating both sides with respect to time from $0$ to $t$, noticing that
  \begin{align*}
  	&\quad\int_{\S^1}\l(1-\cT\cdot\cT_h \r)(0)q_h(0) \d\xi
  	= \frac{1}{2}\int_{\S^1}\l|\cT-\cT_h \r|^2(0)q_h(0)\d\xi\\
  	&\le C\int_{\S^1}|\p_\xi X-\p_\xi X_h |^2(0)\d \xi\le C\|\p_\xi(X-I_hX )(0) \|_{L^2}^2 \le Ch^2\|X^0\|_{H^2}^2,
  \end{align*}
we are led to the estimate \eqref{cons} with appropriate constant $C$ by applying Gronwall's inequality  and  Sobolev embedding $H^1(\S^1)\hookrightarrow L^\infty(\S^1)$.
\end{proof}

\smallskip

\begin{lemma}\label{Control lemma}
Suppose
   \[
    \int_{\S^{1}}\left|\cT-\cT_h\right|^{2}q_h\d \xi+\|q-q_h\|_{L^2}^{2}\leq C_1h^{2},\quad \forall\ t\in[0,T^*],
   \]
then  there exists a constant $h_0$ such that for any $0<h\le h_0$, we have
   \[
  \inf_{\xi}q_h\ge 3\kappa_1/4,\quad \text{and}\quad \sup_{\xi}q_h\le 3\kappa_2/2,\quad \forall\ t\in[0,T^*],
   \]
   where the constant $h_0$ depends on $C_1,C_p, C_P,\kappa_1, \kappa_2$ and $K(X)$.

\end{lemma}
\smallskip
\begin{proof} Applying the triangle inequality, the interpolation error estimate \eqref{Interpolation estimate}, and the inverse estimate \eqref{Inverse estimate}, we can derive
	\begin{align*}
		\l\|\p_\xi X-\p_\xi X_h\r\|_{L^\infty}
		&\le \l\|\p_\xi X- I_h\p_\xi X\r\|_{L^\infty}+\l\|\p_\xi X_h- I_h\p_\xi X\r\|_{L^\infty}\\
		&\le Ch^{1/2}\|\p_\xi X\|_{H^1}+Ch^{-1/2}\l\|\p_\xi X_h-I_h\p_\xi X\r\|_{L^2}\\
		%&\le Ch^{1/2}\|X\|_{H^2}+Ch^{-1/2}\l(\l\|\p_\xi X_h-\p_\xi X\r\|_{L^2}+\l\|\p_\xi X- I_h \p_\xi X\r\|_{L^2} \r)\\
		&\le Ch^{1/2}\|X\|_{H^2}+Ch^{-1/2}\l\|\p_\xi X_h-\p_\xi X\r\|_{L^2}.
	\end{align*}
The assumption and equality \eqref{Basic equality} imply
  \[\|\p_\xi X-\p_\xi X_h\|_{L^2}
  \le \Big(\int_{\S^1}q q_h|\mathcal{T}-\mathcal{T}_h|^{2}\d \xi\Big)^{1/2}+\|q-q_h\|_{L^2}\le \sqrt{C_1}\big(1+\l\|\p_\xi X\r\|_{L^\infty}^{1/2}\big) h.
  \]
  Then it follows that
  \begin{align*}
  	\l\|\p_\xi X-\p_\xi X_h\r\|_{L^\infty} \le C\l( K(X),C_1\r)h^{1/2},
  \end{align*}
and the conclusion follows by recalling Assumption 2.2.
\end{proof}

\smallskip

In order to estimate the length element difference, we first give the following preliminary  lemma by following the lines of \cite[Lemma 3.2, Lemma 4.1]{Pozzi-Stinner}.
\begin{lemma}\label{Preparation lemma}
Given the assumptions of Lemma \ref{Consistency estimate lemma}, then there exists a constant $C$ depending on $\kappa_1,C_p,C_P, c_0,C_0,T$  such that the following estimates hold for $t\in [0,T^*]$:
  \begin{align}
  	|R-R_j| &\le C\l(|L-L_h|+|\cT-\cT_j|+ |\cT-\cT_{j+1}|\r),\quad j=1,\ldots, N,\label{Preparation estimate 1}\\
  	%|\p_t X-R| &\le C\l(1+K(X)\r),\label{Preparation estimate 2}\\
  	\int^t_0 (q_j+q_{j+1})&|\dot X_j-R_j |^2 \d s \le Ch,\quad j=1,\ldots, N,\label{Preparation estimate 3}
  \end{align}
  where $R$ and $R_j$ are defined as \eqref{rrj}.
  Moreover, we have the estimates on $I_j$:
  \begin{equation}\label{Preparation estimate 4}
		\begin{split}
	&\sum\limits_{k=j-1}^{j+1}\|\cT-\cT_k\|_{L^2(I_j)}^2\le Ch^2\|X\|^2_{H^2(S_j)} +C\|\cT-\cT_h \|_{L^2(S_j)}^2, \\
			&\sum\limits_{k=j-1}^j\|\p_t X-\dot X_k\|_{L^2(I_j)}^2\le Ch^2\|\p_tX\|^2_{H^1(I_j)}+C\|\p_t X-\p_t X_h\|_{L^2(I_j)}^2,\\
&\sum\limits_{k=j-1}^{j+1}\|qh_j-q_k\|_{L^2(I_j)}^2\le Ch^4\|X\|^2_{H^2(S_j)}+ Ch^2\|q-q_h\|_{L^2(S_j)}^2,
		\end{split}
	\end{equation}
with $S_j=I_j\cup I_{j+1}\cup I_{j-1}$.	
\end{lemma}

\begin{proof}
Firstly \eqref{Preparation estimate 1} can be easily derived by employing Assumption 2.1 and \eqref{Control assumption, consistency estimate}:
		\begin{align*}
		|R-R_j|
		&=\Big| -\frac{2\pi}{L}\cN+\frac{2\pi}{L_h}\frac{\cN_{j}q_j+\cN_{j+1}q_{j+1}}{q_j+q_{j+1}}\Big|\\
		&\le \Big|-\frac{2\pi}{L}\cN+\frac{2\pi}{L_h}\cN \Big|+\Big|\frac{2\pi}{L_h}\frac{q_j\l(\cN -\cN_j \r) }{q_j+q_{j+1}} \Big|+\Big|\frac{2\pi}{L_h}\frac{q_{j+1}\l(\cN -\cN_{j+1} \r)}{q_j+q_{j+1}} \Big|\\
		&\le C\l(|L-L_h|+|\cT-\cT_j|+ |\cT-\cT_{j+1}|\r).
	\end{align*}
%The estimate (\ref{Preparation estimate 2}) is a direct  consequence of (\ref{Control assumption, consistency estimate}) and Sobolev embedding theory.
For \eqref{Preparation estimate 3}, equation \eqref{dq2} implies
   \begin{align*}
  	&\int^t_0 (q_j+q_{j+1})|\dot X_j-R_j |^2+ (q_j+q_{j-1})|\dot X_{j-1}-R_{j-1} |^2 \d s\\
  	&=4\int^t_0\cT_j\cdot\l(R_j-R_{j-1}  \r) -\frac{\d }{\d t}q_j\  \d s\\
  	&\le 4\int^t_0 \cT_j\cdot \frac{2\pi}{L_h}\Big(-\frac{\cN_{j}q_j+\cN_{j+1}q_{j+1}}{q_j+q_{j+1}}+\frac{\cN_{j}q_j+\cN_{j-1}q_{j-1}}{q_j+q_{j-1}} \Big)\ \d s+4q_j(0)\\
  	&\le  C\int^t_0  \Big|\cT_j\cdot \frac{\cN_{j+1}q_{j+1}}{q_j+q_{j+1}}\Big| +\Big|\cT_j\cdot \frac{\cN_{j-1}q_{j-1}}{q_j+q_{j-1}}\Big| \ \d s+Ch\\
  	&\le  \eps \int^t_0\frac{|\cT_{j+1}-\cT_{j}|^2}{q_j+q_{j+1}} +\frac{|\cT_{j-1}-\cT_{j}|^2}{q_j+q_{j-1}}  \ \d s + C(\eps)\int^t_0\frac{q_{j+1}^2}{q_j+q_{j+1}} +\frac{q_{j-1}^2}{q_j+q_{j-1}} \ \d s +Ch\\
  	&\le  \eps \int^t_0\frac{|\cT_{j+1}-\cT_{j}|^2}{q_j+q_{j+1}} +\frac{|\cT_{j-1}-\cT_{j}|^2}{q_j+q_{j-1}}  \ \d s +C(\eps)h\\
  &=\frac{\eps}{4}\int^t_0 (q_j+q_{j+1})|\dot X_j-R_j |^2+ (q_j+q_{j-1})|\dot X_{j-1}-R_{j-1} |^2 \d s+C(\eps)h,
  \end{align*}
where for the second inequality we used \eqref{Control of perimeter} and   \eqref{Interpolation estimate} to get that
  \[
  \frac{2\pi}{L_h}\le C,\quad q_j(0)=h_j|\partial_\xi X^0_h |=h_j|\partial_{\xi} I_h X^0 |\le C(X)h,
  \]
for the third inequality we employed Young's inequality and the fact $\mathcal{T}_{j}\cdot \mathcal{N}_j=0$ to derive
  \begin{align*}
            \left| \mathcal{T}_j\cdot \frac{\mathcal{N}_{j+1}q_{j+1} }{q_j+q_{j+1}}\right|
                &=\left|(\mathcal{T}_j-\mathcal{T}_{j+1})\cdot \frac{\mathcal{N}_{j+1}q_{j+1} }{q_j+q_{j+1}}  \right|\\
                &\le |\mathcal{T}_j-\mathcal{T}_{j+1} |\frac{q_{j+1}}{q_j+q_{j+1}}\le \varepsilon\frac{|\mathcal{T}_j-\mathcal{T}_{j+1} |^2}{q_j+q_{j+1}}+C(\varepsilon)\frac{q_{j+1}^2}{q_j+q_{j+1}},
            \end{align*}
and  for the last equality we used \eqref{Semidiscrete, another formulation}. Obviously \eqref{Preparation estimate 3} follows by taking $\eps=1$. The estimates in (\ref{Preparation estimate 4}) can be established by using similar arguments as in \cite[Lemma 4.1]{Pozzi-Stinner} and are deleted here for brevity. \end{proof}

\smallskip

Next we present the key length difference estimate with the aid of Lemma \ref{Preparation lemma}.
\begin{lemma}\label{Norm difference estimate}
Given the assumptions of Lemma \ref{Consistency estimate lemma}, then we have
	\[
		\|q-q_h\|_{L^2}^2\leq C\int_{0}^{t} \int_{\S^{1}}\left|\p_t X-\p_t X_h\right|^{2}q_h \d \xi \d s+C\int^t_0\int_{\S^{1}}\left|\cT-\cT_h\right|^{2}q_h\d \xi\d s+Ch^{2},
	\]
	where $C$ is a  constant depending on $C_p,C_P,\kappa_1,c_0,C_0,T^*$ and $K(X)$.
\end{lemma}

\smallskip
\begin{proof}
%The proof is of  same streamline in \cite[Lemma 4.4]{Pozzi-Stinner} and we only focus on the estimates involving $L$ and $L_h$.
By definition, one has
\[\int_{\S^{1}}\l(q(\xi, t)-q_h(\xi, t)\r)^{2} \mathrm{~d} \xi=\sum\limits_{j=1}^N\int_{I_j}(q(\xi,t)-q_j(t)/h_j)^2d\xi.\]
	By integration for $\frac{\d }{\d t}\l(h_jq-q_j \r)$, we can write
\[
	(h_jq-q_j)(t)=(h_jq-q_j)(0)+\int^t_0 \l(h_j\p_tq-\dot{q}_j \r)\ \d s=:P +\int^t_0 A\ \d s,
\]
where by the interpolation error estimate \eqref{Interpolation estimate} and inverse estimate \eqref{Inverse estimate}, $P$ satisfies
\[|P|=h_j\big||\p_\xi X^0|-|\p_\xi I_h X^0|\big|_{I_j}\le Ch\|\p_\xi(X^0-I_hX^0)\|_{L^\infty(I_j)}\le Ch^{3/2}\|X^0\|_{H^2(I_j)}.\]
%\begin{align*}
%&\le Ch\l(\|\p_\xi X^0-I_h \p_\xi X^0\|_{L^\infty(I_j)}+
%\|I_h \p_\xi X^0- \p_\xi I_h X^0\|_{L^\infty(I_j)}\r)\\
%&\le Ch^{3/2}\|X^0\|_{H^2(I_j)}+Ch^{1/2}\|I_h \p_\xi X^0- \p_\xi I_h X^0\|_{L^2(I_j)}\\
%&\le Ch^{3/2}\|X^0\|_{H^2(I_j)}+Ch^{1/2}\l(\|I_h \p_\xi X^0- \p_\xi X^0\|_{L^2(I_j)}+\| \p_\xi (X^0-I_h X^0)\|_{L^2(I_j)}\r)
%\end{align*}
Applying \eqref{Length element equation} and \eqref{dq2}, on each grid element $I_j=[\xi_{j-1}, \xi_j]$, we can write $A$ as
\begin{align*}
	A
	= \frac{\d }{\d t}\l(h_jq-q_j \r)&=-\Big(\frac{h_jq}{2}|\p_t X-R|^2-\frac{q_j+q_{j+1}}{4}|\dot X_j-R_j|^2 \Big)\\
		&\quad -\Big(\frac{h_jq}{2}|\p_t X-R|^2-\frac{q_j+q_{j-1}}{4}|\dot X_{j-1}-R_{j-1}|^2 \Big)\\
		&\quad -\l(h_j q(\p_t X-R )\cdot R+\cT_j\cdot\l(R_j-R_{j-1}  \r) \r)\triangleq:-A^+-A^- -\hat A.
\end{align*}
The terms $\int^t_0 |A^+|\ \d s$ and $\int^t_0 |A^-|\ \d s$ are estimated in  \cite[Lemma 4.4]{Pozzi-Stinner}, which read as
\begin{align*}
&\int^t_0 (|A^+|+|A^-|)\d s\le CQ_j
	+C h\Big(\int^t_0 \sum\limits_{k=j-1}^j|\p_t X-R-(\dot X_k -R_k)|^2 \d s \Big)^{1/2},\\
&Q_j:=\Big(\int^t_0 \l(|h_jq-q_j|^2+|h_jq-q_{j+1}|^2+|h_jq-q_{j-1}|^2 \r) \d s \Big)^{1/2},\end{align*}
where \eqref{Preparation estimate 3} has been used. Applying \eqref{Preparation estimate 1}, we immediately get
\begin{align*}
&\int^t_0 (|A^+|+|A^-|)\d s	\le CQ_j+ChT_j+ChY_j+Ch\Big(\int^t_0 |L-L_h |^2\ \d s \Big)^{1/2},\\
&T_j:=\Big(\int^t_0 \sum\limits_{k=j-1}^{j+1}|\cT-\cT_k |^2 \d s \Big)^{1/2},\quad Y_j:=\Big(\int^t_0 |\p_t X-\dot X_j |^2+|\p_t X-\dot X_{j-1} |^2\d s \Big)^{1/2}.
\end{align*}
It remains to estimate $\int^t_0 |\hat A|\ \d s$. By definition, one has
\begin{align*}
		\hat A
		%&=h_j q(\p_t X-R )\cdot \Big(-\frac{2\pi}{L}\cN \Big)+\cT_j\cdot \frac{2\pi}{L_h}\Big(-\frac{\cN_{j}q_j+\cN_{j+1}q_{j+1}}{q_j+q_{j+1}}+\frac{\cN_{j}q_j+\cN_{j-1}q_{j-1}}{q_j+q_{j-1}} \Big)\\
		&=\frac{h_j q}{2}(\p_t X-R )\cdot \Big(-\frac{2\pi}{L}\cN \Big)- \frac{2\pi}{L_h}\frac{q_{j+1}}{q_j+q_{j+1}}\cT_j\cdot\cN_{j+1}\\
		&\quad +\frac{h_j q}{2}(\p_t X-R )\cdot \Big(-\frac{2\pi}{L}\cN \Big)+ \frac{2\pi}{L_h}\frac{q_{j-1}}{q_j+q_{j-1}}\cT_j\cdot\cN_{j-1}\triangleq:\hat A_1 +\hat A_2.
	\end{align*}
Recalling \eqref{Semidiscrete, another formulation}, we observe
\[
		\cT_j\cdot\cN_{j+1}=\cT_j\cdot \l(\cT_{j+1}-\cT_j \r)^\perp=-\cN_j\cdot \l(\cT_{j+1}-\cT_j \r)=-\frac{q_j+q_{j+1}}{2}\cN_j\cdot  \big(\dot X_j-R_j \big),\]
which implies
	\begin{align*}
		\hat A_1
		%&=\frac{h_j q}{2}(\p_t X-R )\cdot \Big(-\frac{2\pi}{L}\cN \Big)+\frac{\pi}{L_h} q_{j+1}\cN_j\cdot   \big(\dot X_j-R_j \big)\\
		&=h_j q\l(\p_t X-R \r)\cdot \Big(-\frac{\pi}{L}\cN+\frac{\pi}{L_h}\cN_j \Big)\\
		&\quad +(q_{j+1}-h_j q)\l(\p_t X-R  \r)\cdot \frac{\pi}{L_h}\cN_j +\big(\big(\dot X_j-R_j \big)-\l(\p_t X-R  \r) \big)\cdot \frac{\pi}{L_h} q_{j+1}\cN_j.
	\end{align*}
	Therefore, by the assumptions and \eqref{Preparation estimate 1}, we can estimate
\begin{align*}
	&\int^t_0 |\hat A_1| \d s
	\le Ch\int^t_0 \Big|\frac{\pi \cN}{L}-\frac{\pi\cN_j }{L_h} \Big|+|\dot X_j-R_j-\p_t X+R|\d s+C\int^t_0|q_{j+1}-h_j q|\d s\\
	&\le Ch\int^t_0|L-L_h|+\sum\limits_{k=j}^{j+1}|\cT-\cT_k| \d s+C\int^t_0|q_{j+1}-h_j q|\d s +Ch\int^t_0 |\dot X_j-\p_t X|\d s\quad \\
	&\le Ch\Big(\int^t_0 |L-L_h |^2\ \d s \Big)^{1/2}+CQ_j +ChT_j +ChY_j,
\end{align*}
and similar estimates can be established for $\int^t_0 |\hat A_2|\ \d s$.

To summarize, we obtain the following estimate on $I_j$
\[
|h_jq-q_j|
	\le  Ch^{3/2}\|X^0\|_{H^2(I_j)} + Ch\Big(\int^t_0\| q-q_h\|_{L^2}^2\d s \Big)^{1/2}+CQ_j +ChT_j +ChY_j,\]
where we have used \eqref{Estimate of perimeter}.
Applying   \eqref{Preparation estimate 4}, we get
\begin{align*}
	&\|h_jq(t)-q_j(t)\|_{L^2(I_j)}^2
	\le Ch^4\Big(\|X^0\|_{H^2(I_j)}^2+\int^t_0\|\p_t X\|^2_{H^1(I_j)}+\| X\|^2_{H^2(S_j)}\d s\Big)\\
	& \,+ Ch^2\int^t_0h\|q-q_h\|_{L^2}^2+\|q-q_h\|_{L^2(S_j)}^2+\|\cT-\cT_h\|_{L^2(S_j)}^2+
\|\p_t X-\p_t X_h\|_{L^2(S_j)}^2\d s.
\end{align*}
Summing up over all grid elements $I_j$ yields 
\[Ch^2\|q-q_h\|_{L^2}^2
	\le C h^4+Ch^2\int^t_0\| q-q_h\|_{L^2}^2 +\|\cT-\cT_h\|_{L^2}^2+\|\p_t X-\p_t X_h\|_{L^2}^2\d s,
\]
where we  have used the inequality
\[\|h_jq(t)-q_j(t)\|_{L^2(I_j)}^2= h_j^2\|q-q_h\|_{L^2(I_j)}^2 \ge Ch^2\|q-q_h\|_{L^2(I_j)}^2.\]
Finally, the desired estimate is concluded by a Gronwall's argument.
\end{proof}

We are now in a position to prove Theorem \ref{Error estimate}.

\smallskip

\emph{Proof of Theorem \ref{Error estimate}.}
Since the nonlinear terms in \eqref{Semidiscrete, lumped mass} are locally Lipschitz with respect to $X_j$, the local existence and uniqueness is guaranteed  by standard ODE theory. Let $T^*\in (0,T)$ be the maximal time such that the semi-discrete solution $X_h$ exists and the following estimates hold
 	\be\label{Xhb}
  	\inf\l|\p_{\xi} X_{h}\r|\ge \kappa_1/2,\quad \text{and}\quad \sup\l|\p_{\xi} X_{h}\r|\le 2\kappa_2 ,\quad \forall\ t\in [0,T^*].
  \ee
 Combining Lemma \ref{Consistency estimate lemma}, Lemma \ref{Norm difference estimate} and employing Gronwall's argument, we can yield that for any $t\in [0,T^*]$, it holds
  \be\label{mpf1}
  \int_{0}^{t} \int_{\S^{1}}\left|\p_t X-\p_t X_h\right|^{2}q_h \d \xi \d s+\sup_{0\le s\le t} \int_{\S^{1}}\left|\cT-\cT_h\right|^{2}q_h\d \xi \le Ch^2.
  \ee
Plugging this back into Lemma \ref{Norm difference estimate}, we obtain for any $t\in [0,T^*]$,
  \be\label{mpf2}
   \|q-q_h\|_{L^2}^2\leq Ch^{2}.
  \ee
By Lemma \ref{Control lemma}, there exists $h_0>0$ depending  on  $C_p, C_P, \kappa_1, \kappa_2,T$ and $K(X)$ such that for any $0<h\le h_0$, we have
  \[
\inf|\partial_{\xi} X_{h}|\ge 3\kappa_1/4,\quad \text{and}\quad \sup|\partial_{\xi} X_{h}|\le 3\kappa_2/2,\quad \text{at }t=T^*.
  \]
  By standard ODE theory, we can uniquely  extend the above semi-discrete solution in a neighborhood of $T^*$. And  by continuity, we obtain
   \[
   \inf|\partial_{\xi} X_{h}|\ge \kappa_1/2,\quad \text{and}\quad \sup|\partial_{\xi} X_{h}|\le 2\kappa_2,\quad \text{in a neighborhood of } T^*.
  \]
This contradicts to the maximality of $T^*$, and thus $T^*=T$. Thus \eqref{Xhb}-\eqref{mpf2} hold for $t\in [0, T]$ and the nondegeneration property \eqref{Nondegenerate of edge length}  follows by \eqref{Xhb} and Assumption 2.1 by noticing $q_j=h_jq_h$. We derive the error estimate by integration and \eqref{Basic equality}:
\begin{align*}
&\quad\|X(\cdot, t)-X_h(\cdot, t)\|_{H^1}^2=
\int_{\S^1}|X-X_h|^2\d\xi+\int_{\S^1}|\p_\xi X-\p_\xi X_h|^2\d\xi\\
&\le 2\int_{\S^1}\big(\int_0^t \p_t X-\p_t X_h\d s\big)^2\d\xi+2\|X^0-I_h X^0\|_{L^2}^2+\|q-q_h\|_{L^2}^2+\int_{\S^1}|\mathcal{T}-\mathcal{T}_h|^2qq_h \d\xi\\
&\le 2\int_{\S^1} T\int_0^t |\p_t X-\p_t X_h|^2\d s\d\xi+Ch^2
\le Ch^2,
\end{align*}
and the proof is completed.
$\square\hfill$

\section{Numerical results}
\label{sec:illust}

In this section we present a fully discrete version of \eqref{Semidiscrete, weak} to simulate the AP-CSF. Choose an integer $m$, set the time step $\tau=T/m$ and $t_k=k\tau$, $k=0,\ldots,m$. For simplicity we choose a uniform mesh, i.e., $\xi_j=jh$ for $j=0, \ldots, N$ and $h=2\pi/N$. We take $X_h^0=I_h X^0$. For $k\ge 1$, find  $X_h^k=\sum\limits_{j=1}^NX^k_j\varphi_j\in V_h$ by
		\begin{align*}
			&\int_{\S^1}\left|\partial_{\xi} X_{h}^{k-1}\right| \delta_{\tau} X_{h}^{k} \cdot v_{h} \d \xi+\int_{\S^1} \partial_{\xi} X_{h}^{k} \cdot \partial_{\xi} v_{h}/\left|\partial_{\xi} X_{h}^{k-1}\right| \d \xi\\
			&+\int_{\S^1}h^2|\p_\xi X^{k-1}_h|\p_\xi\delta_\tau X^k_h\cdot \p_\xi v_h/6\ \d \xi+\int_{\S^1}2\pi (\p_\xi X_h^k)^\perp\cdot v_h/L_h^{k-1} \d \xi=0, \quad \forall\ v_{h} \in V_h,
		\end{align*}		
where $\delta_\tau$ is the backward finite difference $\delta_\tau X_h^m=(X_h^m-X_h^{m-1})/\tau$, and $L_h^{k-1}$ is the length of the image of $X_h^{k-1}$. Or it can be written equivalently as a discretization for the ODE system \eqref{Semidiscrete, lumped mass}:
\[	\frac{q_j^{k-1}+q_{j+1}^{k-1}}{2\tau}(X_j^k-X_j^{k-1})-\frac{X^k_{j+1}-
X^k_{j}}{q^{k-1}_{j+1}}+\frac{X^k_{j}-X^k_{j-1}}{q^{k-1}_{j}}+
\frac{\pi}{L_h^{k-1}}\l(X^k_{j+1}-X^k_{j-1} \r)^\perp=0.
			\]

First, we test the convergence rates in the $L^2$ norm, $H^1$ seminorm and the error of velocity, respectively. Since the exact solution of the AP-CSF \eqref{AP-CSF} is unknown, we consider the following numerical errors
\begin{align*}
                \left(\mathcal{E}_1\right)_{h,\tau}(T)
                &:= \max_{1\le k\le T/\tau} \big\|X^k_{h,\tau}-X_{h/2,\tau/4}^{4k}\big\|_{L^2(\mathbb{S}^1)},\\
                \left(\mathcal{E}_2\right)_{h,\tau}(T)
                &:= \max_{1\le k\le T/\tau} \big\|\partial_\xi X^k_{h,\tau}-\partial_\xi X_{h/2,\tau/4}^{4k}\big\|_{L^2(\mathbb{S}^1)},\\
                \left(\mathcal{E}_3\right)_{h,\tau}(T)
                &:=\Big(\sum_{k=0 }^{T/\tau-1}\tau \Big\|\frac{X^{k+1}_{h,\tau}-X^{k}_{h,\tau}}{\tau}
-\frac{X^{4k+1}_{h/2,\tau/4}-X^{4k}_{h/2,\tau/4}}{\tau/4}\Big\|_{L^2(\mathbb{S}^1)}^2\Big)^{1/2},
            \end{align*}
where $X_{h,\tau}^k$ represents the solution obtained by the above fully discrete scheme with mesh size $h$ and time step $\tau$. The corresponding convergence order is defined as:
            \[
        \text{Order}_i=\log\Big(\frac{\left(\mathcal{E}_i\right)_{h,\tau}(T)}{\left(\mathcal{E}_i\right)_{h/2,\tau/4}(T)} \Big)\Big/ \log 2,\quad i=1,2,3.
            \]

 The errors and convergence orders are displayed in Table \ref{tab:convergence order}, where we choose $h=2\pi/N$ and $\tau=0.5h^2$ and the initial value is given by $X^0(\xi)=(2\cos \xi, \sin \xi)$. The results indicate that the numerical solution converges linearly in space in the $H^1$ seminorm, which agrees with the theoretical analysis in Theorem \ref{Error estimate}. We can also observe that the solution and the velocity converge quadratically in $L_t^\infty L_x^2$ and $L_t^2L_x^2$, respectively, which is superior than the result in Theorem \ref{Error estimate}.
\begin{table}[h!]
\def\temptablewidth{1\textwidth}
\vspace{-12pt}
\caption{Numerical errors up to $T=1/4$.}\label{tab:convergence order}
{\rule{\temptablewidth}{1pt}}
\begin{tabular*}{\temptablewidth}{@{\extracolsep{\fill}}lllllll}
\quad\multirow{2}{1cm}{$N$} &  \multirow{2}{1cm}{$\l(\cE_1\r)_{h,\tau}(1/4)$}  & \multirow{2}{1cm}{$\mathrm{Order}_1$} &  \multirow{2}{1cm}{$\l(\cE_2\r)_{h,\tau}(1/4)$}  & \multirow{2}{1cm}{$\mathrm{Order}_2$} &  \multirow{2}{1cm}{$\l(\cE_3\r)_{h,\tau}(1/4)$}  & \multirow{2}{1cm}{$\mathrm{Order}_3$}  \\ \\ \hline
   \quad16 &  2.08E-2   & -     &  1.15E-0   & -    &  3.09E-2   & - \\
   \quad32 &   5.42E-3  & 1.94  &   6.01E-1  & 1.94 &   1.01E-2  & 1.61\\
    \quad64 &  1.37E-3  & 1.99  &  3.03E-1  & 0.99  &  2.76E-3    & 1.87  \\
   \quad128 &   3.42E-4 & 2.00  &   1.52E-1 & 1.00  &   7.09E-4 & 1.96  \\
  \end{tabular*}
{\rule{\temptablewidth}{1pt}}
%\medskip
\end{table}

%\begin{table}[htbp!]
%\def\temptablewidth{1\textwidth}
%\vspace{-12pt}
%\caption{Numerical errors up to $T=1/4$.}\label{tab:convergence order}
%{\rule{\temptablewidth}{1pt}}
%\begin{tabular*}{\temptablewidth}{@{\extracolsep{\fill}}lllll}
%\multirow{2}{2cm}{$N$} & \multirow{2}{2cm}{$\l(\cE_1\r)_{h,\tau}(1/4)$} & \multirow{2}{2cm}{\text{Order}}  & \multirow{2}{2cm}{$\l(\cE_2\r)_{h,\tau}(1/4)$}  & \multirow{2}{2cm}{\text{Order}}   \\ \\ \hline
%   8  & 8.68E-1  & - & 9.67E-1 &  -\\
%   16 & 4.37E-1   & 0.99 & 5.29E-1  & 0.87  \\
%   32 &  2.19E-1  &  1.00  & 2.97E-1  & 0.83  \\
%    64 &  1.10E-1  & 1.00 & 1.67E-1  & 0.83  \\
%   128 &  5.48E-2 &  1.00 & 9.17E-2 & 0.87  \\
%    256 & 2.74E-2  & 1.00 & 4.84E-2 & 0.92   \\
% \end{tabular*}
%{\rule{\temptablewidth}{1pt}}
%\medskip
%\end{table}

\begin{figure}[h!]
\hspace{-10mm}
\begin{minipage}[c]{0.8\linewidth}
\centering
\includegraphics[width=14cm,height=4.5cm]{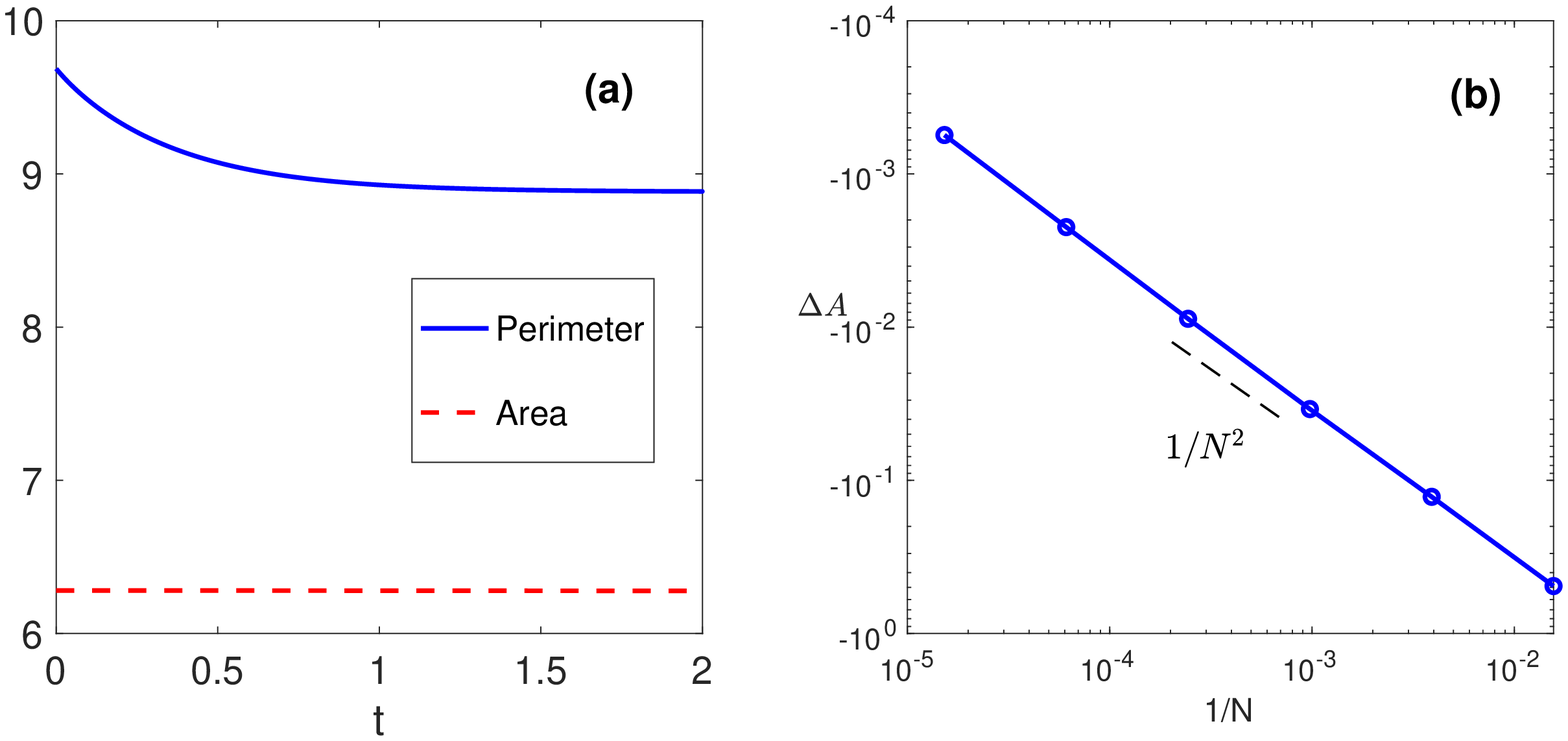}
\end{minipage}
		\caption{Numerical results for an initial ellipse curve (i.e., $\frac{x^2}{4}+y^2=1$): (a) evolution of the perimeter and area; (b) the asymptotic area loss at $T=1/4$.}
		\label{The preservation of geometric structures}
	\end{figure}

\begin{figure}[h!]
\vspace{-5mm}
%\begin{minipage}[c]{0.7\linewidth}
\centering
	\includegraphics[width=12cm,height=4.0cm]{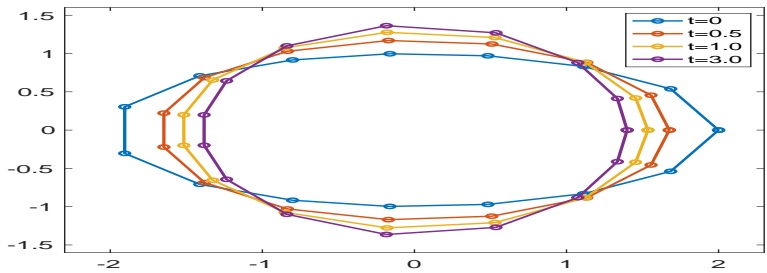}
%\end{minipage}
		\caption{The evolution of an initial convex polygon under the AP-CSF.}
		\label{The evolution of convex polygon}
	\end{figure}

Finally, we check the structure-preserving properties of our algorithm. As is shown in Fig.~\ref{The preservation of geometric structures}(a), the length of the curve is decreasing during the evolution, which confirms the theoretical analysis in Theorem \ref{Perimeter decreasing property}. Furthermore, Fig.~\ref{The preservation of geometric structures}(a) shows the area is almost preserving and more specifically, Fig.~\ref{The preservation of geometric structures}(b) indicates that the area enclosed by the curve has an error at $O(h^2)$. The evolution of the polygon with the number of grid points $N=15$, which approximates the evolution of the ellipse determined by $x^2/4+y^2=1$, is shown in Fig.~\ref{The evolution of convex polygon}, from which we clearly see that the polygon keeps convex during the evolution, which verifies the convexity-preserving property in Theorem \ref{Preservation of convexity}.

%Our algorithm  can also apply to non-convex initial data, for instance, the evolution of an initial flower shape is shown in Fig. \ref{The evolution of flower}, which gradually collapses to a circle. This example shows that our scheme also works well for non-convex curves.
%\begin{figure}[htbp]
%		\includegraphics[width=14cm]{Figure2.eps}
%		\caption{The evolution of a flower shape: (a) t=0, (b) t=0.10, (c) t=0.15, (d) t=0.30, (e) t=0.50, (f) t=1.00.}
%		\label{The evolution of flower}
%	\end{figure}

\smallskip

%=============================================================================
%                                                                          REFERENCES
% =============================================================================

%\bibliographystyle{siam}
%
%
%\bibliography{biblio}

\end{document}